\documentclass[12 pt]{article}%
\usepackage{amsmath, amsfonts, amsthm, color,latexsym, mathtools}
\usepackage{amsmath, ulem}
\usepackage{amsfonts}
\usepackage{amssymb}
\usepackage{color, soul}
\usepackage[all]{xy}
\usepackage{graphicx}%
\providecommand{\U}[1]{\protect\rule{.1in}{.1in}}
\allowdisplaybreaks[4]
\newtheorem{theorem}{Theorem}[section]
\newtheorem{proposition}[theorem]{Proposition}
\newtheorem{corollary}[theorem]{Corollary}
\newtheorem{example}[theorem]{Example}
\newtheorem{examples}[theorem]{Examples}

\newtheorem{lemma}[theorem]{Lemma}
\newtheorem{final remark}[theorem]{Final Remark}
\newtheorem{definition}[theorem]{Definition}
\allowdisplaybreaks[4]

\usepackage{wrapfig}
\usepackage{pgf,tikz}
\usetikzlibrary{decorations.pathreplacing,quotes,arrows,matrix,decorations.fractals,shapes.geometric,calc,positioning,intersections}
\usepackage{graphicx,float}

\usepackage[utf8]{inputenc}
\usepackage[english]{babel}
\usepackage{pifont}

\usepackage{lipsum}

\usepackage{amsmath,amsthm,amsfonts,mathtools, ulem}
\usepackage{amssymb,amstext,amscd}
\usepackage{latexsym,bm,mathrsfs,esint} 
\usepackage{extarrows}

\usepackage{url}
\usepackage[fixlanguage]{babelbib}

\usepackage[nottoc]{tocbibind}
\newcommand{\R}{\mathbb{R}}

\usepackage[inline]{enumitem}
\setlist[enumerate,1]{label=(\roman*)}

\begin{document}

\title{A unifying approach to closed subspaces of linear and multilinear operators }
\author{Geraldo Botelho\thanks{Supported by FAPEMIG grants RED-00133-21 and APQ-01853-23}~ and Ariel Mon\c c\~ao\thanks{Supported by a CAPES scholarship. \newline 2020 Mathematics Subject Classification: 46B42, 46B50,  47B65, 46G25, 47H60 47L22.\newline Keywords: Banach spaces, Banach lattices, sequential continuity, factorable multilinear operators.}}
\date{}
\maketitle

\begin{abstract} We prove several abstract results giving general conditions under which subspaces of linear or multilinear operators on Banach spaces or Banach lattices are closed. Each of these abstract results is followed by concrete applications, concerning classes of linear/multilinear operators already studied as well as new classes.
\end{abstract}

\section{Introduction}
\indent\indent Let $E, E_1, \ldots, E_m,F$ be Banach spaces, let ${\cal L}(E,F)$ be the Banach space of bounded linear operators from $E$ to $F$ endowed with its usual norm, and let ${\cal L}(E_1, \ldots, E_m;F)$ be the Banach space of continuous $m$-linear operators from $E_1 \times \cdots \times E_m$ to $F$ also endowed with its usual norm. Every analyst knows that closed subspaces of ${\cal L}(E,F)$ and of ${\cal L}(E_1, \ldots, E_m;F)$  play a key role in Banach space theory (and also in Banach lattice theory), specifically in the study of classes of linear/multilinear operators with special properties. Many of such analysts have already come across a proof that a  certain subspace of linear/multilinear operators is closed which is closely related to the proof that some other subspace is closed. Actually, it is not uncommon to see the same argument being used to prove closedness of several different subspaces. The idea of this paper is to avoid, or at least to reduce substantially, such unnecessary repetitions. To do so, we tried to figure out the common features that make a certain argument effective in the proof that subspaces of operators enjoying such features are closed. This effort resulted in several abstract results we prove along this manuscript.

Each abstract result proved in this manuscript is followed by concrete applications that either: recover subspaces that have been proved to be closed in the literature; or prove that a subspace that has been claimed to be closed is indeed closed; or prove that subspaces consisting of operators that have already been studied, but which closedness had not been investigated thus far, are closed; or, finally, prove that subspaces consisting of operators that have not been considered yet are closed.

The paper is organized as follows. In Section 2 we prove a quite general result on closed subspaces of linear operators from a normed space to a locally convex space defined, or characterized, by the transformation of certain subsets of the domain space onto relatively compact subsets of the target space. In Section 3 we prove two results: (i) Sets of multilinear operators whose linear adjoints belong to a closed subspace of linear operators are closed subspaces. (ii) Sets of multilinear operators between Banach lattices whose linearizations on the positive projective tensor product belong to a closed subspace of linear operators are closed with respect to the regular norm. In Section 4 we investigate subspaces of multilinear operators defined, or characterized, by the transformation of sequences with special properties into convergent sequences. Section 5 is devoted to subspaces of multilinear operators that enjoy a certain property separately, that is, the linear (or multilinear) operators obtained by fixing variables of the multilinear operator enjoy such property. In Section 6 we prove that subspaces of multilinear operators that factor through closed subspaces of linear operators and closed subspaces of multilinear operators are closed. In Section 7 we just state two further simple results.

Banach spaces are supposed to be real or complex, and Banach lattices are always supposed to be real. In particular, if we are referring to a linear operator from a Banach space to a Banach lattice (or vice-versa) or between Banach lattices, both of the spaces are supposed to be real. Unless explicitly stated otherwise, ${\cal L}(E,F)$ and ${\cal L}(E_1, \ldots, E_m;F)$ are endowed with the usual operator norm. If the underlying spaces are Banach lattices, by ${\cal L}_r(E,F)$ and ${\cal L}_r(E_1, \ldots, E_m;F)$ we denote the spaces of regular linear/multilinear operators. The symbol $B_E$ stands for the closed unit ball of $E$ and $E^{*}$ denotes its topological dual. By $J_E$ we mean the canonical isometric embedding from $E$ to its bidual $E^{**}$, which is also a Riesz homomorphism if $E$ is a Banach lattice.  The notion of bounded subset of a Banach spaces is canonical; but to avoid ambiguity with the notion of order boundedness, a subset of a Banach lattice that is bounded with respect to its norm shall be referred to as a norm bounded set. In the finite cartesian product of normed space we always consider the maximum norm $\|\cdot\|_\infty$, and, if the spaces are Banach lattices, we always consider the coordinatiwise ordering.

For Banach space theory we refer to \cite{Botelho, megginson}, for Banach lattice theory to \cite{Alip, Meyer-Peter}, for (spaces of) multilinear operators to \cite{Din, mujica}, and for (spaces of) regular multilinear operators on Banach lattices to \cite{Buskes3, loane}. Operators ideals, in particular closed operator ideals, are taken in the sense of \cite{df, Pietsch}, and ideals of multilinear operators, in particular closed ideals of multilinear operators (or closed multi-ideals) are taken in the sense of  \cite{fg, pietsch83}.

\section{Closed subspaces of linear operators}

\indent\indent Some well studied classes of linear operators are defined, or characterized, by the transformation of certain sets in the domain space onto relatively norm compact sets in the target space. For instance:\\
$\bullet$ Compact operators between Banach spaces send bounded sets onto relatively norm compact sets.\\
$\bullet$ Completely continuous operators (or Dunford-Pettis operators) between Banach spaces send weakly compact sets onto relatively norm compact sets.\\
$\bullet$ $AM$-compact operators from a Banach lattice to a Banach space send order intervals onto relatively norm compact sets \cite[Section 5.4]{Alip}.\\
$\bullet$ $DW$-compact operators from a Banach lattice to a Banach spaces send disjoint weakly compact sets onto relatively norm compact sets \cite{Jin} (for details, see Example \ref{exe-compact}(iii)).

The spaces of all such operators are closed in the space of bounded linear operators; furthermore, the proofs are all similar. Our first result gives a unified way to guarantee that spaces of operators defined, or characterized, in this way are always closed. We prove a quite general result in the setting of (real or complex) topological vector spaces. Some terminology is needed. The following concepts can be found in \cite{Bourbaki, Schaefer2}:\\
$\bullet$ A topological vector space $E$ is {\it quasi-complete} if every bounded closed subset of $E$ is complete with respect to the uniform structure induced by $E$ \cite[1.4, p.\,16]{Schaefer2}.\\ 
$\bullet$ A subset $A$ of a topoogical vector space $E$ is said to be:\\
\indent$\circ$ {\it Totally bounded} if for every neighborhood $V$ of the origin there exist finitely many vectors $a_1, \ldots, a_n \in A$ such that $A \subseteq \bigcup\limits_{i=1}^n (a_i + V)$.  \\
\indent$\circ$ {\it Relatively compact} if the closure $\overline{A}$ of  $A$ in $E$ is compact.

A topological vector space is {\it locally convex} if the origin has a neighborhood basis consisting of convex sets, or equivalently, if its topology is generated by a family of seminorms \cite[Corollary 8.3.8]{Botelho}. In what follows all locally convex spaces are supposed to be Hausdorff.

Let $E$ be a  normed space and let $F$ be a locally convex space whose topology is generated by a family of seminorms $\{p_i\colon i\in I\}$. The {\it topology of uniform convergence on bounded sets} on the space ${\cal L}(E,F)$ of continuous linear operators from $E$ to $F$ is the topology generated by the family of seminomrs $\{P_i : i \in I\}$, where, for each $i \in I$, $P_i(T)=\sup\{p_i(T(x))\colon \|x\|\leq1\}$ for every $T \in {\cal L}(E,F)$. When $F$ is also normed space, this topology coincides with the usual norm topology on ${\cal L}(E,F)$ \cite[p,\,81]{Schaefer2}. All we need to know about this topology is the following:

\begin{lemma}\label{on3k} If $W$ is a neighborhood of $0 \in F$, then the set $V:= \{S \in {\cal L}(E,F) : S(x) \in W \mbox{ for every } x \in B_E\}$ is a neighborhood of $0 \in {\cal L}(E,F)$ in the topology of uniform convergence on bounded sets.
\end{lemma}

\begin{proof} The set $W$ contains in a basic neighborhood of $0 \in F$, so there are $\varepsilon > 0$ and $i_1, \ldots, i_k \in I$ such that
\begin{equation}\label{o92n}\{y \in F : p_{i_j}(y) < \varepsilon,\, j = 1, \ldots, k\} \subseteq W. \end{equation}
For this same $\varepsilon$, the set
$$U:= \{S \in {\cal L}(E,F) : P_{i_j}(S) < \varepsilon,\, j = 1, \ldots, k\} $$
is a (basic) neighborhood of $0 \in {\cal L}(E;F)$. It is enough to show that $U \subseteq V$.  Given $S \in U$, for $j = 1, \ldots, k$, we have
$$ \sup_{x \in B_E} p_{i_j}(S(x))= P_{i_j}(S) < \varepsilon$$
By (\ref{o92n}) we obtain that $S(x) \in W$ for every $x \in B_E$, that is, $S \in V$.
\end{proof}

We say that a property $\mathfrak{P}$ of bounded subsets of a given normed space is a {\it cone property} if  $\lambda  C$ has $\mathfrak{P}$ whenever $\lambda > 0$ and $C$ has $\mathfrak{P}$.

\begin{theorem}\label{P-rel-comp} Let $E$ be a  normed space, let $F$ be a  quasi-complete locally convex space, and let $\mathfrak{P}$ be a cone property of bounded subsets of $E$. Then the set of operators that send subsets of $E$ having $\mathfrak{P}$ onto relatively compact subsets of $F$ is a closed subspace of ${\cal L}(E,F)$ in the topology of uniform convergence on bounded sets.
\end{theorem}

\begin{proof}  In this proof we are always considering, on ${\cal L}(E,F)$, the topology of uniform convergence on bounded sets. 
We have to prove that
$$M:= \{ T \in {\cal L}(E,F) : T(C) \mbox{ is relatively compact  whenever } C \mbox{ has } \mathfrak{P}\}$$ is a closed subspace of ${\cal L}(E,F)$. The fact that $M$ is a linear subspace follows immediately from the following two facts:\\
$\bullet$ A subset of a quasi-complete Hausdorff topological vector space is totally bounded if and only if it is relatively compact \cite[pp.\,25,27]{Schaefer2}.\\
$\bullet$ If $A$ and $B$ are totally bounded subsets of a topological vector space  and $\lambda$ is a scalar, then $A + \lambda B$ is totally bounded as well \cite[5.1 p.\,25]{Schaefer2}.

Let $T \in \overline{M}$ and let $C \subseteq E$ with property $\mathfrak{P}$ be given. The assumption that $\mathfrak{P}$ is a cone property allows us to assume, without loss of generality, that $C \subseteq B_E$. For an arbitrary neighborhood $V$ of $0 \in F$, by \cite[Proposition 8.1.7]{Botelho} there is a balanced neighborhood of the origin $W_1$ such that $W_1 + W_1 \subseteq V$. By the same result, there exists a further neighborhood of the origin $W$ such that $W + W \subseteq W_1$. As $0 \in W$, $W \subseteq W_1$, so $W + W + W \subseteq W_1 + W_1 \subseteq V$. 
By Lemma \ref{on3k} the set
$$Z:=\{S \in {\cal L}(E,F) : S(x) \in W \mbox{ for every } x \in B_E\}$$ is a neighborhood of $0 \in {\cal L}(E,F)$. Using that translations are homeomorphisms on topological vector spaces \cite[Proposition 8.1.3]{Botelho}, it follows that $T+Z$ is a neighborhood of $T$ in ${\cal L}(E,F)$. Then $(T+Z)\cap M \neq \emptyset$ because $T \in \overline{M}$, say $S \in (T+Z) \cap M$, that is, $S \in M$ and there is $S_1 \in Z$ such that $S = T + S_1$. Therefore, $T(x)-S(x) =S_1(x) \in W$ for every $x \in B_E$. In particular, $T(x)-S(x) \in W$ for every $x \in C$. Since $S \in M$, the set $S(C)$ is totally bounded, hence there are finitely many vectors $b_1, \ldots, b_n \in S(C)$ such that $ S(C) \subseteq \bigcup\limits_{i=1}^n (b_i + W)$. Using that $W$ is balanced and $w\coloneqq T(x) - S(x) \in W$, we get $S(x)=T(x)+(-w)\in (T(x)+W)$ and $T(x)=S(x)+w\in (S(x)+W)$ for every $ x \in C$. Then, $S(C) \subseteq T(C) + W$, from which we have $b_i \in a_i + W$ for some $a_i \in T(C)$, $i=1, \ldots, n$; and $T(C) \subseteq S(C) + W$. Therefore,
\begin{align*}
T(C) &\subseteq S(C) + W \subseteq \left( \bigcup_{i=1}^n (b_i + W) \right) + W \subseteq \bigcup_{i=1}^n (b_i + W + W) \\
&\subseteq \bigcup_{i=1}^n (a_i + W + W + W) \subseteq \bigcup_{i=1}^n (a_i + V).
\end{align*}
Since the neighborhood $V$ of $0$ in $F$ is arbitrary, this proves that $T(C)$ is totally bounded, hence it is relatively compact because $F$ is quasi-complete. This shows that $T \in M$ and completes the proof.
\end{proof}

\begin{examples}\label{exe-compact}\rm All properties of bounded sets we shall consider in the examples below are cone properties. The (easy) proofs are omitted.

\medskip

\noindent(i) It is obvious that Theorem \ref{P-rel-comp} recovers the fact that subspaces of compact operators and of completely continuous operators between Banach spaces are closed.

\medskip

\noindent(ii) The fact that being an order interval in a Banach lattice is a cone property assures that Theorem \ref{P-rel-comp} recovers the fact that the subspace of $AM$-compact operators from a Banach lattice $E$ to a Banach space $F$ (see the definition at the beginning of this section) is a closed subspace of ${\cal L}(E,F)$.

\medskip

\noindent(iii) According to \cite{Jin}, a norm bounded subset $A$ of a Banach lattice $E$ is {\it disjoint weakly compact} if every disjoint sequence $(x_n)_n$ in ${\rm sol}(A)\coloneqq\{x\in E\colon |x|\leq |y|\text{ for some }y\in A\}$ is weakly null. Recently, in \cite{Jin} the authors defined a linear operator from a Banach lattice to a Banach space to be {\it $DW$-compact} if it sends disjoint weakly compact sets to relatively norm compact sets. Since being a disjoint weakly compact set is a cone property, the subspace of all $DW$-compact operators from a Banach lattice $E$ to a Banach space $F$ is a closed subspace of ${\cal L}(E,F)$ by Theorem \ref{P-rel-comp}. This fact is neither proved nor mentioned in \cite{Jin}.

\medskip

\noindent(iv) It is rather canonical to say that a bounded subset $A$ of a Banach space is a {\it Dunford-Pettis set} if $\sup\limits_{x\in A}|x^*_n(x)|\longrightarrow0$ for every weakly null sequence $(x^*_n)_n$ in $E^*$ (see, e.g., \cite{lewis}). For $1\leq p\leq\infty$, according to \cite{Alikhani} a bounded subset $A$ of a Banach space $E$ is {\it weakly $p$-compact} if every sequence $(x_n)_n$ in $A$ admits a subsequence $(x_{n_k})_k$ such that $(x^*(x_{n_k}-x))_k\in\ell_p$ for some $x \in A$ and every functional $x^*\in E^*$.

 In \cite{Alikhani} the author defined an operator $T$ between Banach spaces to be {\it pseudo weakly compapct of order $p$}, $1\leq p\leq\infty$, if $T$ sends Dunford-Pettis weakly $p$-compact sets to relatively compact sets. In \cite{Alikhani} it is mentioned, without proof, that the subspace consisting of such operators is closed. Since being a Dunford-Pettis weakly $p$-compact set is a cone property, this fact follows from Theorem \ref{P-rel-comp}.

\medskip

\noindent(v) Let us see three examples that illustrate perfectly the philosophy of this paper, in the sense that the results we prove are intended to avoid unnecessary repetitions of the same argument. First we recall that a norm bounded subset of a Banach lattice is {\it $L$-weakly compact} if every disjoint sequence $(x_n)_n$ in ${\rm sol}(A)$ is norm null (see, e.g., \cite[Defintion 3.6.1]{Meyer-Peter}). According to \cite[Definition 1]{Alpay}, a norm bounded subset of a Banach lattice $E$ is {\it $b$-order bounded} if $A$ is order bounded in $E^{**}$. Of course, we are identifying a subset of $E$ with its range in $E^{**}$ under the canonical isometric isomorphism $J_E: E \longrightarrow E^{**}$. A linear operator $T$ from a Banach lattice to a Banach space is said to be:\\
$\bullet$ {\it $b$-AM-compact} if $T$ sends  $b$-order bounded sets to relatively compact sets \cite{Aqzzouz}.\\
$\bullet$ {\it $LW$-weakly compact} if $T$ sends  $L$-weakly compact sets to relatively compact sets \cite{Hajji}.\\
$\bullet$ {\it $b$-AM-Dunford-Pettis} if $T$ sends $b$-order bounded weakly compact sets to relatively compact sets \cite{Baklouti}.

In each of these three papers, the authors prove that the space consisting of the corresponding operators is closed. Also, the three proofs are quite similar. Being a $b$-order bounded set and being an $L$-weakly compact set are cone properties, so the closedness of these three spaces follows, simultaneously, from Theorem \ref{P-rel-comp}.
\end{examples}

\section{Closed subspaces of multilinear operators: first results}

\indent\indent In this section we present some general results that provide tools for those who wish to prove that a certain subspace of the space of continuous multilinear operators is closed. The range of applications is large, we restrict ourselves to multilinear operators that have already been investigated in the literature which corresponding subspaces have not been proved to be closed.

 Let $E_1, \ldots, E_m, F$ be normed spaces. According to the classical approach due to Aron and Schottenloher \cite{Aron}, 
 the (topological) adjoint of an $m$-linear operator $A \in \mathcal{L}(E_1,\dots,E_m;F)$ is the following bounded linear operator: 
\begin{equation}\label{lm2d} A^*\colon F^*\to \mathcal{L}(E_1,\dots,E_m;\mathbb{R})~,~ A^*(y^*)(x_1,\dots,x_m)=y^*(A(x_1,\dots,x_m)).\end{equation}

The next result follows immediately from the well known fact that the correspondence $A \in \mathcal{L}(E_1,\dots,E_m;F) \mapsto A^* \in {\cal L}(F^*, \mathcal{L}(E_1,\dots,E_m;\mathbb{K}))$ is an isometric isomorphism.

\begin{proposition}\label{fecho-adjunto}
Let $E_1,\dots,E_m$ be normed spaces and let $F$ be a Banach space. If $M$ is  a closed subset of $\mathcal{L}(F^*,\mathcal{L}(E_1,\dots,E_m;\mathbb{K}))$, then the set $M_*:=\{A\in\mathcal{L}(E_1,\dots,E_m;F):A^* \in M\}$ is a closed subset of $\mathcal{L}(E_1,\dots,E_m;F)$. In particular, if $M$ is a closed subspace, then $ M_*$ is a closed subspace as well.
\end{proposition}

\begin{examples}\rm \label{exe-L-M} (i) An $m$-linear operator $A \in \mathcal{L}(E_1,\dots,E_m;F)$ is {\it compact} ({\it weakly compact}) if $A$ sends bounded subsets of $E_1\times\cdots \times E_m$ to compact (weakly compact) subsets of $F$. Using that spaces of compact and of weakly compact linear operators are closed (for the compact case see Example \ref{exe-compact}(i), for the weakly compact case see \cite[Corollary 3.4.10]{megginson}), and that $A$ is compact (weakly compact) if and only if $A^*$ is compact (weakly compact) \cite{Aron, mu, ry},  Proposition \ref{fecho-adjunto} recovers the facts that spaces of compact and of weakly compact multilinear operators are closed.

\medskip

\noindent(ii) Let $E_1,\dots,E_m$ be Banach spaces and let $F$ be a Banach lattice. According to \cite[Definition 4.1]{Ariel}, a continuous $m$-linear operator $A\colon E_1\times\cdots\times E_m\to F$ is {\it $L$-weakly compact} if $A(B_{E_1 \times \cdots \times E_m})$  is an $L$-weakly compact subset of $F$ (cf. Example \ref{exe-compact}(v)). Recall that a linear operator from a  Banach lattice to a Banach space is {\it $M$-weakly compact} if it sends norm bounded disjoint sequences to norm null sequences \cite[Definition 5.59]{Alip}. In \cite[Theorem 4.5]{Ariel} it is proved that an operator $A \in {\cal L}(E_1, \ldots, E_m;F)$ is $L$-weakly compact if and only if its adjoint $A^*$ is an $M$-weakly compact linear operator. Since the subspace of $M$-weakly compact operators is closed in $\mathcal{L}(F^*,\mathcal{L}(E_1,\dots,E_m;\mathbb{K}))$ \cite[Theorem 5.65]{Alip}, the subspace of $m$-linear $L$-weakly compact operators is closed in ${\cal L}(E_1, \ldots, E_n;F)$ by Proposition \ref{fecho-adjunto}. This fact was neither proved nor mentioned in \cite{Ariel}.
\end{examples}

For the next results we recall the positive projective tensor product of Banach lattices. We start by calling $E_1\overline{\otimes}\cdots\overline{\otimes}E_m$ the Fremlin tensor product of the Archimedean Riesz spaces $E_1, \ldots, E_m$ \cite{Buskes3, Fremlin2,Fremlin1}. If $E_1, \ldots, E_m$ are Banach lattices, then the functional
$$x \in E_1\overline{\otimes}\cdots\overline{\otimes}E_m \mapsto \|x\|_{|\pi|}=\inf \left\{ \sum_{j=1}^k \|x_j^1\| \cdots \|x_j^m\| : x_j^i\in {E_i}^+,\, |x| \leq \sum_{j=1}^k x_j^1 \otimes \cdots \otimes x_j^m \right\}
$$
defines a lattice norm on $E_1 \overline{\otimes} \cdots \overline{\otimes} E_m$  \cite{Fremlin1}.  The completion of $E_1 \overline{\otimes} \cdots \overline{\otimes} E_m$ with respect to the norm $\|\cdot\|_{|\pi|}$, which is a Banach lattice \cite[Theorem 4.2]{Alip} containing $E_1 \otimes \cdots \otimes E_m$ as a dense subspace \cite[p.\,850, (g)]{Buskes3}, is denoted by  $E_1{\hat{\otimes}}_{|\pi|}\cdots{\hat{\otimes}}_{|\pi|} E_m$ and called the {\it positive projective tensor product of} $E_1, \ldots, E_m$. By
$$\otimes\colon E_1\times\cdots\times E_m\to E_1{\hat{\otimes}}_{|\pi|}\cdots{\hat{\otimes}}_{|\pi|} E_m~,~\otimes(x_1,\dots,x_m)=x_1\otimes\cdots\otimes x_m,$$
we denote the canonical norm one Riesz $m$-morphism.

Henceforth in this section, $E, E_1, \ldots, E_m,F$ are Banach lattices with $F$ Dedekind complete. For every operator $A \in \mathcal{L}_r(E_1,\dots,E_m;F)$ there exists a unique $A^\otimes \in \mathcal{L}_{r}(E_1{\hat{\otimes}}_{|\pi|}\cdots{\hat{\otimes}}_{|\pi|} E_m;F)$ such that $A = A^\otimes \circ \otimes$. Moreover, the correspondence
\begin{equation}A \in \mathcal{L}_r(E_1,\dots,E_m;F) \mapsto A^\otimes \in \mathcal{L}_{r}(E_1{\hat{\otimes}}_{|\pi|}\cdots{\hat{\otimes}}_{|\pi|} E_m;F) \label{km22}
\end{equation}
is an isometric isomorphism, where both of the spaces are endowed with their respective regular norms, and a Riesz isomorphism \cite[Proposition 3.3]{Buskes3}. The next result follows immediately from this isometric Riesz isomorphism.

\begin{proposition} If $H$ is a subspace of $\mathcal{L}_r(E_1,\dots,E_m;F)$ such that $\{A^\otimes : A \in H \}$ is a closed subspace of $\mathcal{L}(E_1{\hat{\otimes}}_{|\pi|}\cdots{\hat{\otimes}}_{|\pi|} E_m,F)$ with the regular norm, then $H$ is closed in $\mathcal{L}_r(E_1,\dots,E_m;F)$ with the regular norm.
\end{proposition}

\begin{example}\label{pm3w}\rm As we have already remarked, the subspace of continuous $M$-weakly compact linear operators is closed in $\mathcal{L}(E_1{\hat{\otimes}}_{|\pi|}\cdots{\hat{\otimes}}_{|\pi|} E_m,F)$. Using that $\|\cdot\| \leq \|\cdot\|_r$, the space of regular $M$-weakly compact linear operators is closed in $\mathcal{L}_r(E_1{\hat{\otimes}}_{|\pi|}\cdots{\hat{\otimes}}_{|\pi|} E_m,F)$  with the regular norm. From (\ref{km22}) and the proposition above it follows that the space of regular $m$-linear operators $A$ such that $A^\otimes$ is $M$-weakly compact is closed in $\mathcal{L}_r(E_1{\hat{\otimes}}_{|\pi|}\cdots{\hat{\otimes}}_{|\pi|} E_m,F)$ with the regular norm. Such operators were used in \cite[Theorem 5.9]{Ariel} to characterize multilinear operators whose adjoints are $L$-weakly compact.
\end{example}

 The order adjoint of an operator $A\in \mathcal{L}_r(E_1,\dots,E_m;F)$, defined as in (\ref{lm2d}), is regular and takes values in the space of regular multilinear forms, that is, $A^*\colon F^*\to \mathcal{L}_{r}(E_1,\dots,E_m;\mathbb{R})$ is a well defined linear operator 
 \cite[Proposition 2.1 and Theorem 2.3]{Ariel}. It is worth mentioning that there exist nor-order bounded (non-regular) linear operators $T$ such that $T^*$ is order bounded (regular), see \cite[p.\,289]{Alip}. Next lemma is an attempt to mitigate this shortcoming. Recall that: (i) If $T \colon E \to F$ is an weakly compact linear operator, then $T^{**}(E^{**})\subseteq J_F(F)$ (Gantmacher's Theorem). (ii) $J_F(F)$ is closed sublattice of $F^{**}$, in particular it is a Banach lattice.

\begin{lemma} \label{lema-fecho-adjunto-r} {\rm (a)} If $T\colon E \to F$ is an weakly compact operator and $T^{**} \colon E^{**} \to J_F(F)$ is regular, then so is $T$. In particular, if $F$ is reflexive and $T^{**}$ is regular, then so is $T$.\\
{\rm (b)} If $F$ is reflexive and $A\colon E_1\times\cdots\times E_m\to F$ is an $m$-linear operator such that $A^*$ is regular, then so is $A$.
\end{lemma}

\begin{proof} (a)  Write $T^{**}=u_1-u_2$, where $u_1, u_2 \in \mathcal{L}(E^{**},J_F(F))$ are positive operators. Since $J_F \colon F \to J_F(F)$ is a Riesz isomorphism, $J_F^{-1} \colon J_F(F) \to F$ is a positive operator \cite[Theorem 2.15]{Alip}. Then $T = (J_F^{-1} \circ u_1 \circ J_E) - (J_F^{-1} \circ u_2 \circ J_E)$ is the difference of two positive operators, that is, $T$ is regular. \\
(b) Write $A^*=T_1-T_2$, where $T_1, T_2 \in \mathcal{L}_r(F^*,\mathcal{L}_r(E_1,\dots,E_m;\mathbb{R}))$ are positive. Then the second adjoint $A^{**}=(A^*)^*=(T_1-T_2)^*=T_1^*-T_2^*$ is regular as well, that is, $A^{**}\in \mathcal{L}_r(\mathcal{L}_r(E_1,\dots,E_m;\mathbb{R})^*,F^{**})$. Taking $F = \R$ in (\ref{km22}) we know that spaces $\mathcal{L}_r(E_1,\dots,E_m;\R)$ and $(E_1{\hat{\otimes}}_{|\pi|}\cdots{\hat{\otimes}}_{|\pi|} E_m)^*$ are isomorphic as Banach lattices by means of the Riesz isomorphism 
$$I\colon \mathcal{L}_r(E_1,\dots,E_m;\R)\to(E_1{\hat{\otimes}}_{|\pi|}\cdots{\hat{\otimes}}_{|\pi|} E_m)^* ~,~ I(B)=B^\otimes.$$
From \cite[Theorem 2.15]{Alip}, both $I$ and $I^{-1}$ are positive, hence $I^*$ and $(I^{-1})^*=(I^*)^{-1}$ are positive as well. Calling on \cite[Theorem 2.15]{Alip} once again, $I^*$ is a Riesz isomorphism from $(E_1{\hat{\otimes}}_{|\pi|}\cdots{\hat{\otimes}}_{|\pi|} E_m)^{**}$ to $\mathcal{L}_r(E_1,\dots,E_m;\R)^*$. It follows that the correspondence 
\[W \in \mathcal{L}_r(\mathcal{L}_r(E_1,\dots,E_m;\R)^*,F^{**}) \mapsto  W \circ I^* \in \mathcal{L}_r((E_1{\hat{\otimes}}_{|\pi|}\cdots{\hat{\otimes}}_{|\pi|} E_m)^{**},F^{**})\] is also an isomorphism of Banach lattices. In particular, $A^{**}$ is identified, up to this correspondence, with the regular operator $A^{**}\circ I^*=(A^\otimes)^{**}$. Indeed, for $T\in (E_1{\hat{\otimes}}_{|\pi|}\cdots{\hat{\otimes}}_{|\pi|} E_m)^{**}$ and $y^*\in F^*$,  \[[A^{**}(I^*(T))](y^*)=(I^*(T))[A^*(y^*)]=T((A^*(y^*))^\otimes))=T((A^\otimes)^*(y^*)) =[(A^\otimes)^{**}(T)](y^*),\] where the second equality follows from $(I^*(T))(B)=T(B^\otimes)$ for every $B\in \mathcal{L}_r(E_1,\dots,E_m;\R)$, and the third follows from an argument that can be found in the proof of \cite[Theorem 4.5]{Ariel}. It follows that 
$(A^\otimes)^{**}$ is an $F^{**}$-valued regular linear operator, thus weakly compact because $F$ is reflexive. It follows from item (a) that $A^\otimes$ is regular as well. By (\ref{km22}) we conclude that $A$ é regular.
\end{proof}

\begin{proposition}\label{fecho-adjunto-r} Suppose that the Banach lattice $F$ is reflexive.
If $G$ is a closed subspace of $\mathcal{L}_r(F^*,\mathcal{L}_r(E_1,\dots,E_m;\R))$ with the regular norm, then the set of $m$-linear operators $A \in \mathcal{L}(E_1,\dots,E_m;F)$ such that $A^*$ belongs to $G$ is a closed subspace of 
 $\mathcal{L}_r(E_1,\dots,E_m;F)$ with the regular norm.
\end{proposition}

\begin{proof} Call $G_*$ the subspace we have to prove is closed. By Lemma \ref{lema-fecho-adjunto-r}(b) we have $G_* \subseteq \mathcal{L}_r(E_1,\dots,E_m;F)$. It is  clear that $G_*$ is a linear subspace. Let $(A_n)_n$ be a sequence in $G_*$ such that $A_n\longrightarrow A$ with respect to the regular norm. Since  $\mathcal{L}_r(E_1,\dots,E_m;F)$ is a Banach lattice \cite[Proposition 3.3]{Buskes3}, $A$ is regular. 
Using that $\|B\| = \|B^*\|$,  $|B^*| \leq |B|^*$ for every regular multilinear operator $B$, and that the regular norm is a lattice norm, we have
\begin{align*}\|A_n^* - A^*\|_r &= \|(A_n - A)^* \|_r=  \||(A_n - A)^*| \|_r\leq \| |A_n - A|^*\|_r = \| |A_n - A|^*\|\\
&=  \| |A_n - A| \| = \|A_n - A\|_r \longrightarrow 0,
\end{align*}
proving that 
$A_n^*\longrightarrow A^*$ with respect to the regular norm. By assumption, $G$ is closed in the regular norm and $A_n^* \in G$ for every $n$, so $A^*\in G$, that is, $A \in G_*$.
\end{proof}

\begin{example}\label{exe-M-L}\rm Calling on \cite[Theorem 5.65]{Alip} and using that $\|\cdot\|\leq\|\cdot\|_r$ like we did in Example \ref{pm3w}, it follows that the subspace of all regular $L$-weakly compact linear operators from $F^*$ to $\mathcal{L}_r(E_1,\dots,E_m;\R)$ is closed in $\mathcal{L}_r(F^*,\mathcal{L}_r(E_1,\dots,E_m;\R))$ with the regular norm. Then, assuming that $F$ is reflexive, by Proposition \ref{fecho-adjunto-r} it follows that the subspace of regular $m$-linear operators whose adjoints are $L$-weakly compact is closed in $\mathcal{L}_r(E_1,\dots,E_m;F)$ with the regular norm. These multilinear operators were applied in \cite[Theorem 5.9]{Ariel}. It is worth recalling that weakly compact operators that fail to be $L$-weakly compact do exist \cite[p.\,322]{Alip}.
\end{example}

\section{Sequential continuity}

\indent\indent Although stated for spaces of multilinear operators, the results of this section also apply to linear opertors, that is, the linear case $m = 1$ of these results provides closed subspaces of linear operators. The first result shows that subspaces of multilinear (or linear) operators that send bounded sequences to norm null sequences are {\it always} closed.

\begin{proposition}\label{fechados-seq-prop} Let $E_1, \ldots, E_m,F$ be Banach spaces and let $\mathfrak{P}$ be a property of bounded sequences in $E_1\times\cdots\times E_m$. Then the set of continuous $m$-linear operators $A \colon E_1 \times \cdots \times E_m \to F$ that send bounded sequences having property $\mathfrak{P}$ to norm null sequences is a closed subspace of ${\cal L}(E_1, \ldots, E_n;F)$.
\end{proposition}

\begin{proof} Let us call $M$ the set we shall prove to be a closed subspace of  ${\cal L}(E_1, \ldots, E_n;F)$. It is plain that $M$ is a linear subspace. Let a sequence $(A_n)_n$ in $M$ such that $A_n \longrightarrow A \in {\cal L}(E_1, \ldots, E_m;F)$ and a sequence $(x^k)_k=((x_1^k,\dots,x_m^k))_k$ in $C\!\cdot\!B_{E_1\times\cdots\times E_m}$ with property $\mathfrak{P}$, where $C>0$, be given. For $\varepsilon > 0$, take $n_0 \in \mathbb{N}$ such that $\|A_{n_0} - A\| < \frac{\varepsilon}{2C^m}$. As $A_{n_0} \in M$, there is $k_0 \in \mathbb{N}$ such that $\|A_{n_0}(x^k)\| \leq \frac{\varepsilon}{2}$ for every $k \geq k_o$. Then
$$\|A(x^k)\| \leq \|(A- A_{n_0})(x^k)\| + \|A_{n_0}(x^k)\| \leq \|A-A_{n_0}\|\!\cdot\!\|x_1^k\|\cdots \|x_m^k\| + \frac{\varepsilon}{2} < \varepsilon  $$
for every $k \geq k_0$. This proves that $A \in M$.
\end{proof}

\begin{examples}\label{im5h}\rm (i) According to \cite{Ribeiro}, a continuous $m$-linear operator $A \colon E_1 \times \cdots \times E_m \to F$ between Banach spaces is a {\it Dunford-Pettis operator} if, for all weakly null sequences $(x^n_i)_n$ in $E_i$, $i=1,\ldots, m$, it holds $A(x^n_1, \ldots, x^n_m) \longrightarrow 0$ in $F$. Using that weakly null sequences are norm bounded, the fact that the subspace of Dunford-Pettis operators is closed in ${\cal L}(E_1, \ldots, E_m;F)$, proved in \cite[Proposition 2.8]{Ribeiro}, is recovered by Proposition \ref{fechados-seq-prop}.

\medskip

\noindent(ii) Given Riesz spaces $E_1, \ldots, E_m$, we say that the vectors $(x_1,\dots,x_m)$, $(y_1,\dots,y_m)\in E_1\times\cdots\times E_m$ are:\\
$\bullet$ {\it Disjoint} if they are disjoint with respect to the usual coordinatewise ordering in $E_1\times\cdots\times E_m$.\\
$\bullet$ {\it Sporadically disjoint} if $x_i$ and $y_i$ are disjoint in $E_i$ for some $i\in\{1,\dots,m\}$ (see \cite{kusraeva}).

Let $E_1, \ldots, E_m$ be Banach lattices and let $F$ be a Banach space. According to \cite{Ariel}, an operator ${\cal L}(E_1, \ldots, E_n;F)$ is {\it $M$-weakly compact} ({\it sporadically $M$-weakly compact}) if $A$ sends norm bounded disjoint (norm bounded sporadically disjoint) sequences in $E_1\times\cdots\times E_m$ to norm null sequences in $F$. By Proposition \ref{fechados-seq-prop}, the subspaces of $M$-weakly compact operators and of sporadically $M$-weakly compact operators are closed in ${\cal L}(E_1, \ldots, E_n;F)$. These facts were neither proved nor mentioned in \cite{Ariel}.

\medskip

\noindent(iii) Now we introduce operators we believe have not been considered yet.  In \cite[Lemma 2.5]{Baklouti} it is proved that a linear operator $T \colon E \to F$ from a Banach lattice to a Banach space is $b$-AM-Dunford-Pettis (see the definition in Example \ref{exe-compact}(v)) if and only if $T$ sends $b$-order bounded weakly null sequences in $E$ to norm null sequences in $F$. Given Banach lattices $E_1, \ldots, E_m$ and a Banach space $F$, now it is a natural definition to say that an operator $A \in {\cal L}(E_1, \ldots, E_n;F)$ is {\it $b$-AM-Dunford-Pettis} if $A(x_1^n, \ldots, x_m^n) \longrightarrow 0$ in $F$ for all $b$-order bounded weakly null sequences $(x_i^n)_n$ in $E_i$, $i = 1,\ldots, m$. From Proposition \ref{fechados-seq-prop} it follows that the subspace of $b$-AM-Dunford-Pettis $m$-linear operators is a closed subspace of ${\cal L}(E_1, \ldots, E_n;F)$.
\end{examples}

On the one hand, for a linear operator $T \colon E \to F$ between Banach spaces and a vector topology $\tau$ em $E$, the following conditions are equivalent:

\noindent\noindent$\bullet$ $T(x_n) \longrightarrow 0$ in $F$ whenever $x_n \stackrel{\tau}{\longrightarrow} 0$ in $E$.\\
$\bullet$ $T(x_n) \longrightarrow T(x)$ in $F$ whenever $x_n \stackrel{\tau}{\longrightarrow}  x$ in $E$.

\medskip

On the other hand, these conditions are no longer equivalent for nonlinear operators. Due to this fact, Proposition \ref{fechados-seq-prop} does not recover the closed subspaces of some classes of multilinear operators that have proved to be useful in the literature. To remedy this situation, we shall prove a more general result that encompasses such classes.

Let $E_1,\dots,E_m$ be Banach spaces. We say that a sequence $((x_1^n,\dots,x_m^n))_n$ in $E_1 \times\cdots\times E_m$ is {\it $m$-bounded} if $\sup\limits_n \|x_1^n\|\!\cdots\!\|x_m^n\| < \infty$. Of course, bounded sequences are $m$-bounded.

We call $\mathfrak{R}$ the subset of $(E_1 \times\cdots\times E_m)^{\mathbb{N}}$ consisting of all $m$-bounded sequences. Let $\sim$ be a relation from $\mathfrak{R}$ to $E_1 \times\cdots\times E_m$, that is, $\sim$ is a subset of  $\mathfrak{R} \times (E_1 \times\cdots\times E_m)$. When the ordered pair $(((x_1^n, \dots,x_m^n))_n, (x_1,\dots,x_m))$ belongs to $\sim$, we write $((x_1^n, \dots,x_m^n))_n \sim (x_1,\dots,x_m)$.

Keeping the notation and the terminology we have just fixed, we have the following:

\begin{proposition}\label{fechados-seq-geral-prop}
 Let $E_1,\dots,E_m$ be Banach spaces. For every Banach space $F$, the set of  multilinear operators  
$A \in {\cal L}(E_1 ,\dots, E_m;F)$ such that $A(x_1^n, \dots,x_m^n) \longrightarrow A(x_1^n, \dots,x_m)$ in $F$ whenever the sequence $((x_1, \dots,x_m^n))_n$ is $m$-bounded and $((x_1^n, \dots,x_m^n))_n \sim (x_1,\dots,x_m)$ 
is a closed subspace of ${\cal L}(E_1 ,\dots, E_m;F)$.
\end{proposition}
\begin{proof} For simplicity, we prove the bilinear case $m=2$, the general case follows analogously. Let us call $\mathfrak{F}$ the set we want to prove is a closed subspace of ${\cal L}(E_1, E_2;F)$. It is obvious that $\mathfrak{F}$ is a linear subspace. To prove that it is closed, let $(A_j)_j$ be a sequence in $\mathfrak{F}$ such that $A_j \longrightarrow A \in {\cal L}(E_1, E_2;F)$, that is, $\|A_j - A\|\longrightarrow 0$. In order to show that $A \in \mathfrak{F}$, let $((x_n,y_n))_n \subseteq E_1 \times E_2$ be a 2-bounded sequence such that $((x_n, y_n))_n \sim (x,y)$. We call $K = \sup\limits_n \|x_n\|\!\cdot\!\|y_n\|$. Given $\varepsilon > 0$ take $j_0 \in \mathbb{N}$ such that $\|A_{j_0} - A\|< \frac{\varepsilon}{3(K + \|x\|\cdot\|y\| + 1)}$.  As $A_{j_0} \in \mathfrak{F}$,  $A_{j_0}(x_n, y_n) \longrightarrow A_{j_0}(x,y)$ in $F$, so there is $j_0 \in \mathbb{N}$ such that $\|A_{j_0}(x_{n}, y_{n})- A_{j_0}(x,y)\| < \frac{\varepsilon}{3}$ for every $n \geq n_0$. Hence,
\begin{align*}\|A(x_n, &y_n) - A(x,y)\| \leq \\
&\leq \|A(x_n, y_n) - A_{j_0}(x_n,y_n)\| + \| A_{j_0}(x_n,y_n)- A_{j_0}(x,y)\| +  \|A_{j_0}(x,y)  - A(x,y)\| \\
&< \|A - A_{j_0}\|\!\cdot\!\|x_n\|\!\cdot\!\|y_n\| + \frac{\varepsilon}{3} + \|A_{j_0} - A\|\!\cdot\!\|x\|\!\cdot\!\|y\| \\
&< \frac{\varepsilon}{3(K + \|x\|\!\cdot\!\|y\| + 1)}\!\cdot \!\|x_n\|\!\cdot\!\|y_n\| + \frac{\varepsilon}{3} + \frac{\varepsilon}{3(K + \|x\|\!\cdot\!\|y\| + 1)}\!\cdot \!\|x\|\!\cdot\!\|y\| < \varepsilon
\end{align*}
for every $n \geq n_0$. This proves that  $A(x_n, y_n) \longrightarrow A(x,y)$, therefore $\mathfrak{F}$ is closed.
\end{proof}

\begin{corollary}\label{fechados-seq-geral-corol}
 Let $E_1,\dots,E_m,F$ be Banach space and, for $k=1,\dots,m$, let $\tau_k$ be a (non necessarily vector) topology on $E_k$ such that $\tau_k$-convergent sequences in  $E_k$ are bounded. Then the set of operators $A \in {\cal L}(E_1,\dots,E_m,F)$ such that $A(x_1^n, \dots,x_m^n) \longrightarrow A(x_1,\dots,x_m)$ in $F$ whenever  $x_k^n \stackrel{\tau_1}{\longrightarrow} x$ in $E_k$ for each $k=1,\dots,m$, is a closed subspace of ${\cal L}(E_1,\dots,E_m,F)$.
\end{corollary}
\begin{proof}
Since bounded sequences in $E_1\times\cdots\times E_m$ are $m$-bounded, it is enough to define
$$((x_1^n, \dots,x_m^n))_n \sim (x_1, \dots,x_m) \Longleftrightarrow x_1^n \stackrel{\tau_1}{\longrightarrow} x_1 \text{, \dots ,} x_m^n \stackrel{\tau_m}{\longrightarrow} x_m $$
and call on Proposition \ref{fechados-seq-geral-prop}.
\end{proof}

\begin{example}\rm Using that weakly convergent sequences in Banach spacess are bounded \cite[Proposition 6.2.5]{Botelho}, the set of multilinear operators $A \in {\cal L}(E_1,\dots,E_m; F)$ that are weakly sequentially continuous, that is, $A(x_1^n, \dots,x_m^n) \longrightarrow A(x_1, \dots,x_m)$ in $F$ whenever $x_k^n \stackrel{w}{\longrightarrow} x_k$ in $E_k$ for each $k=1,\dots,m$, is a closed subspace of ${\cal L}(E_1,\dots,E_m;F)$ by Corollary \ref{fechados-seq-geral-corol}. This is a well known fact, these operators have been studied extensively since the seminal paper \cite{AHV}, both in the theory of nonlinear operators and in infinite dimensional holomorphy (for the general theory, see the monographs \cite{Din, Hajek}, for specific results, see \cite{Aron2, Castillo1, Castillo2}, for recent developments, see \cite{JLPAMS, sheldon, VP}.
\end{example}

Next we give an application of Proposition \ref{fechados-seq-geral-prop}, that does not follow from Corollary \ref{fechados-seq-geral-corol}, which recovers a closed subspace of multilinear operators that have proved to be helpful in the literature.

\begin{example}\rm In \cite{Ewerton2} the authors remark that, from the perspective of the theory of ideals of multilinear operators, the class of weakly sequentially continuous multilinear operators from the previous example does not satisfy some desirable properties (see \cite[Example 4.4]{Ewerton2}). For this reason they studied a new classe, closely related to the class of weakly sequentially continuous operators, which satisfies such desirable properties.

Let $E_1, \ldots, E_m,F$ be Banach spaces. According to \cite[Definition 4.5]{Ewerton2}:\\
$\bullet$ A sequence $((x_1^n,\dots,x_m^n))_n$ in $E_1\times\cdots\times E_m$ {\it converges multilineaarly} to $(x_1, \dots,x_m) \in E_1\times\cdots\times E_m$ if $B(x_1^n, \dots,x_m^n) \longrightarrow B(x_1, \dots,x_m)$ for every continuous $m$-linear form $B \in {\cal L}(E_1,\dots,E_m;\mathbb{K})$. \\
$\bullet$ An operator $A \in {\cal L}(E_1,\dots,E_m;F)$ is {\it multilinearly sequentially continuous} if  $A(x_1^n, \dots,x_m^n) \longrightarrow A(x_1, \dots,x_m)$ in $F$ whenever $((x_1^n, \dots,x_m^n))_n$ converges multilinearly to  $(x_1, \dots,x_m)$.

 In \cite[Theorem 4.8]{Ewerton2} it is proven that the space of multilinearly sequentially continous operators is closed in ${\cal L}(E_1,\dots,E_m;F)$. Proposition  \ref{fechados-seq-geral-prop} recovers  this fact because, as multilinearly convergent sequences are $m$-bounded \cite[Lemma 4.7]{Ewerton2}, it suffices to define
$$((x_1^n, \dots,x_m^n))_n \sim (x_1, \dots,x_m) \Longleftrightarrow ((x_1^n, \dots,x_m^n))_n \mbox{ converges multilinearly to } (x_1, \dots,x_m). $$
\end{example}

Again we finish with an example of multilinear operators that, as far as we know, have not been investigated yet.

\begin{example}\rm In \cite{Peralta} the authors define the {\it Right topology $\tau_R$} on a Banach space $E$ as the locally convex topology on $E$ generated by the seminorms $x \in E \mapsto \|T(x)\|$, where $T$ is any weakly compact linear operator from $E$ to any Banach space. They also define a linear operator $T \colon E \to F$ to be {\it pseudo weakly compact} if $T$ sends Right-null sequences in $E$ to norm null sequences in $F$, that is, $T(x_n) \longrightarrow 0$ in $F$ whenever $x_n \stackrel{\tau_R}{\longrightarrow} 0$ in $E$. These notions have been studied by several authors, for very recent developments see \cite{mik, arXiv2026}. As Right-convergent sequences are weakly convergent, hence bounded, the linear case of Corollary \ref{fechados-seq-geral-corol} recovers the fact that the subspace of pseudo weakly compact operators is closed.

Mimicking the successful definition of weakly sequentially continuous multilinear operators, we say that an operator $A \in {\cal L}(E_1,\dots,E_m;F)$ is {\it Right sequentially continuous} if $A(x_1^n, \dots,x_m^n) \longrightarrow A(x_1, \dots,x_m)$ in $F$ whenever $x_k^n \stackrel{\tau_R}{\longrightarrow} x_k$ in $E_k$ for each $k=1,\dots,m$. Corollary \ref{fechados-seq-geral-corol} gives that the space of Right sequentially continuous operators is closed in ${\cal L}(E_1,\dots,E_m;F)$.
\end{example}

\section{Fixing variables}

 \indent\indent Let $E_1, \ldots, E_m,F$ be Banach spaces. An $m$-linear operator $A \colon E_1 \times \cdots \times E_m \to F$ is continuous if and only if $A$ is {\it separately continuous}, that is, for every $k = 1,\ldots, m$, and all fixed $a_1 \in E_1, \ldots, a_{k-1}\in E_{k-1}, a_{k+1} \in E_{k+1}, \ldots, a_m \in E_m$, the linear operator
$x_k \in E_k \mapsto A(a_1, \ldots, a_{k-1}, x_k, a_{k+1}, \ldots, a_m) \in F $
 is bounded (see, e.g., \cite[Theorem 1.2]{df}). As expected, this result motivated the study of multilinear operators that enjoy a certain property (of linear operators) separately. Spaces of such multilinear operators are the subject of this section.

Just to fix the notation, given a property $\mathfrak{P}$ of bounded linear operators from $E_j$ to $F$, $j = 1, \ldots, m$, we say that an operator $A \in {\cal L}(E_1, \ldots, E_m;F)$ {\it has $\mathfrak{P}$ separately}  if for every $k = 1,\ldots, m$, and all fixed $a_1 \in E_1, \ldots, a_{k-1}\in E_{k-1}, a_{k+1} \in E_{k+1}, \ldots, a_m \in E_m$, the bounded linear operator
$$A_k^{a_{1},\dots,a_{m}}\colon E_k\longrightarrow F~,~ A_k^{a_{1},\dots,a_{m}}(x_k)=A(a_1,\dots,a_{k-1},x_k,a_{k+1},\dots,a_m),$$
has property $\mathfrak{P}$.

\begin{proposition} \label{ic3z}
Let ${\cal C}$ be a subclass of the class of all Banach spaces over $\mathbb{K}$, let $F$ be a Banach space, and let $\mathfrak{P}$ be a property of bounded linear operators from spaces belonging to $\cal C$ to $F$. Suppose that, for every $ E \in {\cal C}$, $\{u \in {\cal L}(E,F): u \mbox{ has } \mathfrak{P}\}$ is a closed subspace of ${\cal L}(E,F)$. Then, for all  $m \in \mathbb{N}$ and $E_1, \ldots, E_m \in {\cal C}$,
$$M: =\{A \in {\cal L}(E_1, \ldots, E_m;F) : A \mbox{ has } \mathfrak{P} \mbox{ separately}\}$$
is a closed subspace of ${\cal L}(E_1, \ldots, E_m;F)$.
\end{proposition}

\begin{proof} Let $m \in \mathbb{N}, E_1, \ldots, E_m \in {\cal C}$ be given. It is easy to see that $M$ is a linear subspace. Suppose that $(A_n)_n$ is a sequence of operators in $M$ such that $A_n \longrightarrow A\in\mathcal{L}(E_1,\dots,E_m;F)$. Fixing $a_2\in E_2,\dots,a_m\in E_m$, each of the linear operators $A_n^{a_2, \ldots, a_m} \colon E_1 \to F$, $n \in \mathbb{N}$ has property $\mathfrak{P}$. For every $x_1 \in E_1$,
$$\|A_n^{a_2, \ldots, a_m}(x_1) - A^{a_2, \ldots, a_m}(x_1)\| = \|(A_n- A)(x_1,a_2, \ldots, a_m)\|\leq \|A_n -A\|\!\cdot\!\|x_1\|\!\cdot\! \|a_2\| \cdots \|a_m\|.$$
Taking the supremum over $x_1  \in B_{E_1}$ and then making $n \longrightarrow \infty$ we get $A_n^{a_2, \ldots, a_m} \longrightarrow A^{a_2, \ldots, a_m}$ in ${\cal L}(E_1,F)$. By assumption, $\{u \in {\cal L}(E_1,F): u \mbox{ has } \mathfrak{P}\}$ is closed in ${\cal L}(E_1,F)$, therefore $A^{a_2, \ldots, a_m}$ has property $\mathfrak{P}$. This proves that $A$ has $\mathfrak{P}$ in the first variable. The cases of the other variables are identical.
\end{proof}

\begin{examples}\rm (i) For the definition of $M$-weakly compact linear operators from a Banach lattice to a Banach space, see Example \ref{exe-L-M}(ii). Multilinear operators that are separately $M$-weakly compact were studied in \cite{Ariel}. Since the set of $M$-weakly compact linear operators from a Banach lattice $E$ to a Banach space $F$ is a closed subspace of ${\cal L}(E,F)$ \cite[Theorem 5.65(1)]{Alip}, from Proposition \ref{ic3z} it follows that, for Banach lattices $E_1, \ldots, E_m$ and a Banach space $F$, the set of separately $M$-weakly compact operators from $E_1 \times \cdots \times E_m$ to $F$ is a closed subspace of $\mathcal{L}(E_1,\ldots,E_m;F)$. This was neither proved nor mentioned in \cite{Ariel}.

\medskip

\noindent(ii) Similarly to the example above, and using \cite[Theorem 5.65(2)]{Alip}, for Banach spaces $E_1, \ldots, E_m$ and a Banach lattice $F$, the set of separately $L$-weakly compact operators from $E_1 \times \cdots \times E_m$ to $F$ is a closed subspace of $\mathcal{L}(E_1,\ldots,E_m;F)$. These multilinear operators were also studied in \cite{Ariel}.
\end{examples}

Instead of fixing only one variable, one can fix several variables of a multilinear operator. This idea has been explored in the theory of coherent ideals of multilinear operators (see \cite{Ribeiro} and references therein). Next result and the subsequent example illustrate how the ideas and examples of this section can be applied to obtain closed subspaces of multilinear operators that satisfy a condition based on fixing several variables. The following symbology/terminology shall simplify the presentation.

Let $A\colon X_1\times\cdots\times X_m\longrightarrow Y$ be an $m$-linear between linear spaces. Fixed $0 \leq k \leq m-1$, $p_1<\cdots<p_k\in\{1,\dots,m\}$ and  $a_{p_1}\in X_{p_1},\dots,a_{p_k}\in X_{p_k}$, it is  clear that the mapping $A_{a_{p_1},\dots,a_{p_k}}\colon X_1\times\stackrel{[p_1,\dots,p_k]}{\dots}\times X_m\longrightarrow Y$ given by \[A_{a_{p_1},\dots,a_{p_k}}(x_1,\stackrel{[p_1,\dots,p_k]}{\dots},x_m)=A(x_1,\dots,x_{p_1-1},a_{p_1},x_{p_1+1},\dots,x_{p_k-1},a_{p_k},x_{p_k+1},\dots,x_m),\]
  where $\stackrel{[p_1,\dots,p_k]}{\dots}$ means that coordinates  $p_1,\dots,p_k$ have been omitted, is an $(m-k)$-linear operator. 
We say that $A_{a_{p_1},\dots,a_{p_k}}$ is the  {\it resulting operator} fixing $a_{p_1},\dots,a_{p_k}$.

Let $\frak{P}$ be a property of $(m-k)$-linear operators from $E_{p_1} \times \cdots \times E_{p_k}$ to $F$, where $0 \leq k \leq m-1$ and $p_1<\cdots<p_k\in\{1,\dots,m\}$. An $m$-linear operator $A \colon X_1 \times \cdots \times X_m \to Y$ is said to {\it has property $\mathfrak{P}$ strongly} if all resulting operators of the type $A_{a_{p_1},\dots,a_{p_k}}$ have property $\mathfrak{P}$.

\begin{proposition}\label{l2dx} Let $E_1, \ldots, E_m,F$ be Banach spaces and let $\frak{P}$ be a property of continuous $(m-k)$-linear operators from $E_{p_1} \times \cdots \times E_{p_k}$ to $F$, where $0 \leq k \leq m-1$ and $p_1<\cdots<p_k\in\{1,\dots,m\}$, such that, for all such $k, p_1, \ldots, p_k$,
$${\cal L}^{\frak{P}}_{p_1, \ldots, p_k}: = \{B \in {\cal L}(E_{p_1}, \ldots, E_{p_k}; F) : B \mbox{ has } \frak{P}\} $$
is a closed subspace of ${\cal L}(E_{p_1}, \ldots, E_{p_k}; F)$.  Then the set of continuous $m$-linear operators $A \colon E_1 \times \cdots \times E_n \to F$ having property $\frak{P}$ strongly is a closed subspace of  ${\cal L}(E_1, \ldots, E_m; F)$.
\end{proposition}

\begin{proof} Let us call $M$ the set we have to prove is a closed subspace. It is clear that $M$ is a linear subspace. Let $0 \leq k \leq m-1$, $p_1<\cdots<p_k\in\{1,\dots,m\}$ and  $a_{p_1}\in E_{p_1},\dots,a_{p_k}\in E_{p_k}$ be given. Now we prove that
$${\cal L}^{\frak{P}}_{a_{p_1}, \ldots, a_{p_k}} := \{A \in {\cal L}(E_1, \ldots, E_m;F) : A_{a_{p_1},\dots,a_{p_k}} \mbox{ has } \frak{P}\} $$
is closed in ${\cal L}(E_1, \ldots, E_n;F)$. It is clear that we can assume $a_{p_j} \neq 0$ for every $j$. Given $A \in \overline{{\cal L}^{\frak{P}}_{a_{p_1}, \ldots, a_{p_k}}}$, for every $\varepsilon > 0$ there is $S \in {\cal L}^{\frak{P}}_{a_{p_1}, \ldots, a_{p_k}}$ such that $\|A - S\| < \frac{\varepsilon}{\|a_{p_1}\|\cdots\|a_{p_k}\|}$. Then, $S_{a_{p_1},\dots,a_{p_k}} \in {\cal L}(E_{p_1}, \ldots, E_{p_k}; F)$ has $\frak{P}$ and
$$\|A_{a_{p_1},\dots,a_{p_k}} - S_{a_{p_1},\dots,a_{p_k}}\|\leq \|A - S\|\!\cdot\! \|a_{p_1}\|\cdots\|a_{p_k}\| < \varepsilon. $$
This shows that $A_{a_{p_1},\dots,a_{p_k}} \in \overline{{\cal L}^{\frak{P}}_{p_1, \ldots, p_k}}= {\cal L}^{\frak{P}}_{p_1, \ldots, p_k}$, that is $A_{a_{p_1},\dots,a_{p_k}}$ has $\frak{P}$. This proves that ${\cal L}^{\frak{P}}_{a_{p_1}, \ldots, a_{p_k}}$ is closed. Therefore,
$$M = \textstyle{\bigcap} \{{\cal L}^{\frak{P}}_{a_{p_1}, \ldots, a_{p_k}} : 0 \leq k \leq m-1,  p_1<\cdots<p_k\in\{1,\dots,m\}, a_{p_1}\in E_{p_1},\dots,a_{p_k}\in E_{p_k}\} $$
is closed as the intersection of closed sets.
\end{proof}

\begin{example}\rm For multilinear operators from the cartesian product of Banach lattices to a Banach space, let $\frak{P}$ be the property of being $M$-weakly compact. In Example \ref{im5h}(ii) we checked that the assumption on ${\cal L}^{\frak{P}}_{p_1, \ldots, p_k}$ in Proposition \ref{l2dx} holds true. Therefore,   for all Banach lattices $E_1,\dots,E_m$ and any Banach space $F$,
the set of strongly $M$-weakly compact $m$-linear operators from $E_1 \times \cdots \times E_m$ to $F$ is a closed subspace of $\mathcal{L}(E_1,\dots,E_m;F)$. Strongly $M$-weakly compact multilinear operators were introduced in \cite{Vinger} and played an important role in the main results proved there.
\end{example}

\section{Factorable multilinear operators}
\indent\indent The paper \cite{pietsch83} is a cornerstone in the study of multilinear operators with special properties, where Pietsch started, among other things, the study of multilinear operators that factor through linear/multilinear operators belonging to certain ideals. More precisely, the idea is to investigate multilinear operators $A \in {\cal L}(E_1, \ldots, E_m;F)$ between Banach spaces which admit a factorization
\[
\begin{tikzpicture}
\matrix (m) [matrix of math nodes,
             row sep=0.5pt,
             column sep=0.5pt] {
E_1 &\times \cdots \times& E_m & \xrightarrow{\quad A \quad} & F \\
\mathllap{\scriptstyle S_{1}}\bigg\downarrow & \vdots&  \mathllap{\scriptstyle S_{m}}\bigg\downarrow &  & \\
G_1 &\times \cdots \times& G_m \\
};

\draw[->] (m-3-3) -- node[midway, below right] {$ \scriptstyle B$}(m-1-5);
\end{tikzpicture}\,.
\]
where $G_1, \ldots, G_m, G_0$ are Banach spaces, $S_1, \ldots, S_m$ are linear operators belonging to certain (possibly different) operator ideals and $B$ is continuous $m$-linear operator with certain properties. In this case we write $A = B \circ(S_1, \ldots, S_m)$, meaning that
\begin{equation}\label{p5m9}A(x_1, \ldots, x_m) = B(S_1(x_1), \ldots, S_m(x_m)) \mbox{ for all } (x_1, \ldots, x_m) \in E_1 \times \cdots \times E_m.
\end{equation}
 Of course, we are interested in the closedness of the subspace consisting of such $m$-linear operators. We shall address, first, the more general case in which $A$ admits a factorization 
\[\label{o2va}
\setlength{\arraycolsep}{2pt}
\begin{array}{ccccc}
E_1 &\times \dots \times& E_m & \xrightarrow{\quad A \quad} & F \\
\mathllap{\scriptstyle S_1}\bigg\downarrow & \vdots&  \mathllap{\scriptstyle S_m}\bigg\downarrow &  & \mathllap{\scriptstyle S_0}\bigg\uparrow  \\
G_{1} &\times \dots \times& G_{m} & \xrightarrow{\quad B \quad} & G_0
\end{array}
\]
where $G_1, \ldots, G_m,G_0$ are Banach spaces, $S_j$ belongs to a certain subspace of ${\cal L}(E_j;G_j), j = 1,\ldots,m$, $B$ belongs to a certain subspace of ${\cal L}(G_1, \ldots, G_m;G_0)$ and $S_0$ belongs to a certain subspace of ${\cal L}(G_0;F)$. Similarly to (\ref{p5m9}), in this case we write $A = S_0 \circ B \circ(S_1, \ldots, S_m)$. Our purpose is to prove that, if all these certain subspaces are closed, then the space of all multilinear operators that admit such a factorization is closed as well. Concrete examples shall appear throughout the section.

We shall need some preparation on $p$-Banach spaces, to which we refer the reader to \cite{kalton}. For $0 < p \leq 1$, a $p$-normed space is a linear space $E$ endowed with a $p$-norm $\|\cdot\|$. In this case, the map
$$(x,y) \in E^2 \mapsto \|x-y\|^p $$
is a metric on $E$. We shall consider $E$ as a metric space endowed with this metric. If this metric space is complete, then $E$ is a $p$-Banach space. In particular, every $p$-Banach space is a complete metrizable (in general non-locally convex) topological vector space. The easy proofs of items (a) and (b) below are omitted (for (a), use \cite[Proposition 5.5.2]{Garling}); for (c) see \cite[Lemma 3.2.5]{livro}.

\begin{lemma} \label{2hjl}{\rm (a)} If $0 < q < p \leq 1$, then every $p$-normed space ($p$-Banach space) is a $q$-normed space ($q$-Banach space).\\
{\rm (b)} Let $q, p \in (0,1]$ and let $(E,\|\cdot\|)$ be a $p$-Banach space. If $\|\cdot\|$ is a $q$-norm on $E$, then $E$ is a $q$-Banach space.\\
{\rm (c)} A $p$-normed space $E$, $0 < p \leq 1$, is a $p$-Banach space if and only if, for every sequence $(x_n)_n$ in $E$ with $\sum\limits_{n=1}^\infty \|x_n\|^p < \infty$, the series $\sum\limits_{n=1}^\infty x_n$ converges in $E$.
\end{lemma}

To give the main result of this section the generality we need, the following definition is necessary.

\begin{definition}\label{ujn7}\rm (a) An {\it injective class of linear operators} on a Banach space $E$ is a correspondence $\mathcal{C}(E)$ that assigns, to each Banach space  $F$, a closed subspace $\mathcal{C}(E,F)$ of $\mathcal{L}(E,F)$ satisfying the following condition: If $T\in \mathcal{C}(E,F)$, $G$ is a Banach space and $I\colon F\to G$ is an injective operator, then $I\circ T\in\mathcal{C}(E,G)$.\\
{\rm (b)} A {\it surjective class of linear operators} on a Banach space $F$ is a correspondence $\mathcal{D}(E)$ that assigns, to each Banach space  $E$, a closed subspace $\mathcal{D}(E,F)$ of $\mathcal{L}(E,F)$ satisfying the following condition: If $T\in \mathcal{D}(E,F)$, $G$ is a Banach space and $P \in {\cal L}(G,E)$ is a surjective operator, then $T \circ P\in\mathcal{D}(G,F)$.\\
{\rm (c)} For $m \in \mathbb{N}$, an {\it $m$-semi-ideal} is a subclass ${\cal N}$ of the class of all continuous $m$-linear operators between Banach spaces such that, for all Banach spaces $E_1,\dots,E_m,F$, the component $\mathcal{N}(E_1,\dots,E_m;F)\coloneqq\mathcal{L}(E_1,\dots,E_m;F)\cap\mathcal{N}$ satisfies:

\noindent (1) $\mathcal{N}(E_1,\dots,E_m;F)$ is a closed subspace of $\mathcal{L}(E_1,\dots,E_m;F)$.

\noindent (2) If $A\in\mathcal{N}(E_1,\dots,E_m;F)$, $u_i\in\mathcal{L}(G_i,E_i), i = 1,\ldots, m$, are surjective operators, and $t\in\mathcal{L}(F,H)$ is an injective operator, then  $t\circ A\circ(u_1,\dots,u_m)\in\mathcal{N}(G_1,\dots,G_m;F)$.\\
(d) Let $\cal N$ be an $m$-semi ideal, let $\mathcal{C}(E_1),\dots,\mathcal{C}(E_m)$ be injective classes of linear operators on the Banach spaces $E_1, \ldots, E_m$, and let $\mathcal{D}(F)$ be a surjective class of linear operators on the Banach space $F$. We say that an operator $A \in {\cal L}(E_1, \ldots, E_m;F)$ belongs to $\mathcal{D}(F)\circ\mathcal{N}\circ(\mathcal{C}(E_1),\dots,\mathcal{C}(E_m))$ if there are Banach spaces $G_0,G_1,\dots,G_m$, linear operators $S_i\in\mathcal{C}(E_i,G_i), i=1,\dots,m$, $S_0\in\mathcal{D}(G_0,F)$, and an $m$-linear operator $B\in\mathcal{N}(G_1,\dots,G_m;G_0)$ such that $A=S_0\circ B\circ(S_1,\dots,S_m)$.
\end{definition}

Given a closed operator ideal ${\cal I}$, of course the correspondence
$E \mapsto {\cal I}(E,F) $ for every $F$ is a injective class of linear operators on $E$, and $F \mapsto {\cal I}(E,F) $ for every $E$ is a surjective class os linear operators on $F$. In this case we shall write ${\cal I}(E)$ and ${\cal I}(F)$ (we believe no ambiguity will arise). Classes of linear operators that do not come from operators ideals will appear in the examples. If $\cal M$ is a closed ideal of multilinear operators, then, for every $m \in \mathbb{N}$, the $m$-linear operators belonging to $\cal M$ form an $m$-semi-ideal.

By ${\rm id}_E$ we mean the identity operator on the Banach space $(E, \|\cdot\|)$. When $E$ is endowed with a norm $\|\cdot\|_1$ different from $\|\cdot\|$, the identity operators $(E, \|\cdot\|) \to (E, \|\cdot\|_1)$ and $(E, \|\cdot\|_1) \to (E, \|\cdot\|)$ shall be denoted by different symbols. As usual, given a sequence $(E_n)_{n=1}^\infty$ of Banach spaces, by $\left(\bigoplus\limits_{n=1}^\infty E_n\right)_1$ we denote the Banach spaces of sequences $(x_n)_n$, $x_n \in E_n$ for every $n \in \mathbb{N}$, such that $\|(x_n)_n\|_1 = \sum\limits_{n=1}^\infty \|x_n\|< \infty$.  In the next statement we employ the notation that has just been fixed in Definition \ref{ujn7}.

\begin{theorem} \label{Prop-8.3}
$\mathcal{D}(F)\circ\mathcal{N}\circ(\mathcal{C}(E_1),\dots,\mathcal{C}(E_m))$ is a closed subspace of $\mathcal{L}(E_1,\dots,E_m;F)$.
\end{theorem}

\begin{proof} This proof was inspired in \cite[Section 8]{Braunss}. For simplicity, we write $\mathcal{A}=\mathcal{D}(F)\circ\mathcal{N}\circ(\mathcal{C}(E_1),\dots,\mathcal{C}(E_m))$.

\medskip

\noindent{\bf Claim 1.} $\cal A$ is a linear subspace of $\mathcal{L}(E_1,\dots,E_m;F)$.

It is clear that $\lambda A \in {\cal A}$ for every $A \in {\cal A}$ and every scalar $\lambda$. Given $A_1, A_2 \in {\cal A}$, there are Banach spaces $G_0,G_1,\dots,G_m,H_0,H_1,\dots,H_m$, linear operators $S_i\in\mathcal{C}(E_i,G_i),T_i\in\mathcal{C}(E_i,H_i),\, i=1,\dots,m$, $S_0\in\mathcal{D}(G_0,F),\,T_0\in\mathcal{D}(H_0,F)$, and $m$-linear operators $B_1\in\mathcal{N}(G_1,\dots,G_m;G_0),\,B_2\in\mathcal{N}(H_1,\dots,H_m;H_0)$ such that $A_1=S_0\circ B_1\circ(S_1,\dots,S_m)$ and $A_2=T_0\circ B_2\circ(T_1,\dots,T_m)$. For each $i=0,\dots,m$, we consider the Banach space  $(G_i\oplus H_i,\|\cdot\|_\infty)$, which is the cartesian product endowed with the maximum norm, the projections
\[P_i\colon G_i\oplus H_i\to G_i~,~P_i(x,y)=x, \mbox{ and }Q_i\colon G_i\oplus H_i\to H_i~,~ Q_i(x,y)=y,\]
and the inclusions
\[J_i\colon G_i\to G_i\oplus H_i~,~ J_i(x)=(x,0), \mbox{ and } U_i\colon H_i\to G_i\oplus H_i~,~U_i(y)=(0,y).\]
Consider also the linear operators $R_0 := S_0\circ P_0 + T_0\circ Q_0 \in {\cal L}(G_0 \oplus H_0, F)$, $R_i := J_i\circ S_i + U_i\circ T_i \in {\cal L}(E_i, G_1 \oplus H_i)$ for $i = 1, \dots, m$, and the multilinear operator
$$B\coloneqq J_0\circ B_1\circ (P_1,\dots,P_m)+U_0\circ B_2 \circ(Q_1,\dots,Q_m) \in {\cal L}(G_1 \oplus H_1, \ldots, G_m \oplus H_m; G_0 \oplus H_0).$$
Since $P_0$ and $Q_0$ are surjective, $S_0\circ P_0$ and $T_0\circ Q_0$ belong to $\mathcal{D}(G_0\oplus H_0,F)$, so does $R_0$. 
Since $J_i$ and $U_i$ are isometric embeddings, $J_i\circ S_i$ 
and $J_i\circ Q_i$ belong to $\mathcal{C}(E_i,G_i\oplus H_i)$, so does $R_i$. 
Since $B_1$ and $B_2$ belong to $\mathcal{N}$, $B\in\mathcal{N}(G_1\oplus H_{1}, \dots ,G_m\oplus H_{m} ; G_0\oplus H_0)$. A simple (and boring) computation shows that the diagram
\[
\setlength{\arraycolsep}{1pt}
\begin{array}{ccccc}
E_1 &\times \dots \times& E_m & \xrightarrow{~ A_1+A_2 ~} & F \\
\mathllap{\scriptstyle R_1}\bigg\downarrow & \vdots&  \mathllap{\scriptstyle R_m}\bigg\downarrow &  & \mathllap{\scriptstyle R_0}\bigg\uparrow  \\
G_1\oplus H_{1} &\times \dots \times& G_m\oplus H_{m} & \xrightarrow{\quad B \quad} & G_0\oplus H_0
\end{array}\,
\]
is commutative, that is, $A_1+A_2=R_0\circ B\circ(R_1,\dots,R_m) \in {\cal A}$.

\medskip

\noindent{\bf Claim 2.} Every $A\in\mathcal{A}$ admits a factorization $A=U_0\circ V\circ(U_1,\dots,U_m)$, with $U_0 \in \mathcal{D}(G; F)$, $V \in \mathcal{N}(G_1, \dots, G_m; G)$, $U_i \in \mathcal{C}(E_i, G_i)$ for $i = 1, \dots, m$, so that $$\|A\|^\frac{1}{m+2}=\|U_0\|=\|V\|=\|U_1\|=\dots=\|U_m\|.$$
The case $A = 0$ is obvious. Given $0 \neq A\in\mathcal{A}$, there is a factorization  $A=T_0\circ B\circ(T_1,\dots,T_m)$ with $0 \neq T_0 \in \mathcal{D}(H; F)$, $0 \neq B \in \mathcal{N}(H_1, \dots, H_m; H)$, $0 \neq T_i \in \mathcal{C}(E_i, H_i)$ for $i = 1, \dots, m$. For all $\alpha,\lambda>0$,
$$A=\frac{T_0}{\alpha \lambda}\circ \alpha B\circ\left(\lambda\|T_2\|\cdots\|T_m\|T_1,\frac{T_2}{\|T_2\|},\dots,\frac{T_m}{\|T_m\|}\right).$$ So, we can choose $\lambda > 0$ such that $\lambda\|T_1\|\cdots\|T_m\|=\|A\|$, and then we can choose $\alpha > 0$ so that $$\left\|\alpha B\circ\left(\lambda\|T_2\|\cdots\|T_m\|T_1,\frac{T_2}{\|T_2\|},\dots,\frac{T_m}{\|T_m\|}\right)\right\|=\|A\|.$$ 
It follows that, in the factorization $A=T_0\circ B\circ(T_1,\dots,T_m)$, we may assume that \[\| T_1 \| = \| B\circ (T_1, \dots, T_m) \| = \| A \| \mbox{ ~and~ } \| T_i \| = 1 \mbox{ for } i = 2, \dots, m.\]
For $h\in H$ we define $\| h \|_0 \coloneqq \max \{ \| h \|, \| T_0(h) \| \}$, which is a norm equivalent to the original norm on $H$ because $T_0$ is a bounded linear operator. So, $G := (H, \| \cdot \|_0)$ is a Banach space. We call $I\colon G\to H$ the identity operator, which is obviously an isomorphism, we define $S_0\colon G\to H$ by $S_0\coloneqq T_0\circ I$ and $B_0\colon H_1\times\cdots\times H_m\to G$ by $B_0=B\circ I$. 
We have $S_0 \in \mathcal{D}(G_0; F)$ and $B_0 \in \mathcal{N}(G_1, \dots, G_m; G)$. From $\| S_0(h) \| = \| T_0(h) \| \leq \| h \|_0$ for every $h \in G$ we get $\| S_0 \| \leq 1$. For all $x_1 \in E_1, \ldots, x_m \in E_m$, using that$ \| B\circ (T_1, \dots, T_m) \| = \| A \|$ we obtain
\begin{align*}
\| B_0(T_1( x_1), \dots, T_m (x_m)) \|_0&= \max \{ \| B \circ (T_1, \ldots, T_m)( x_1, \ldots, x_m)) \|, \| A( x_1, \dots, x_m) \| \}\\& \leq \| A \|\!\cdot\! \| x_1 \| \cdots \| x_m\|,
\end{align*}
which gives $\| B_0\circ (T_1, \dots, T_m) \| \leq \| A \|$.

For each $h_1 \in H_1$,  consider the continuous $(m-1)$-linear operator $B_0\circ(h_1,T_2, \dots, T_m)\colon  E_2\times\cdots \times E_m\to G$ given by $$[B_0\circ(h_1,T_2, \dots, T_m)](x_2,\dots,x_m)=B_0(h_1,T_2(x_2), \dots, T_m(x_m)),$$ and define the norm $\| h_1 \|_1 := \max \{ \| h_1 \|, \| B_0\circ (h_1,T_2, \dots, T_m) \| \}$ on $H_1$, which is equivalent to original norm of $H_1$, to get the Banach space $G_1 := (H_1, \| \cdot \|_1)$. Of course, the identity operator $I_1\colon H_1\to G_1$ is an isomorphism. Defining $S_1=I_1\circ T_1$ and $B_1\colon G_1\times H_2\times\cdots \times H_m\to G$ by $B_1=B_0 \circ (I_1^{-1},{\rm id}_{H_2},\dots,{\rm id}_{H_m})$, we get
$S_1 \in \mathcal{C}(E_1,G_1)$ and $B_1 \in \mathcal{N}(G_1,H_2,\dots, H_m; G)$. From $\| B_1\circ (h_1,T_2,\dots,T_m) \| = \| B_0\circ (h_1,T_2,\dots,T_m) \| \leq \| h_1 \|_1$ it follows that $\| B_1\circ ({\rm id}_{G_1},T_2,\dots,T_m) \| \leq 1$. And
 $$\| S_1(x_1)\|_1 = \max \{ \| S_1(x_1)\|, \| B_0\circ (T_1(x_1),T_2,\dots,T_m) \| \} \leq \| A \|\!\cdot\! \| x_1 \|$$
for every $x_1 \in E_1$ implies $\| S_1 \| \leq \| A \|$.

We define an equivalent norm on $H_2$ by putting $\| h_2 \|_2 := \max \{ \| h_2 \|, \| B_1\circ ({\rm id}_{G_1},h_2,S_3, \dots, S_m) \| \}$, where the operator $ B_1\circ ({\rm id}_{G_1},h_2,S_3, \dots, S_m)$ is defined similarly to what we did above, to get the Banach space $G_2 := (H_2, \| \cdot \|_2)$. Call $I_2\colon H_2\to G_2$ the identity operator, $S_2\coloneqq I_2\circ T_2$ and
$$B_2\colon G_1\times G_2\times H_3\times\cdots H_m\to G~,~B_2=B_1 \circ ({\rm id}_{G_1},I_2^{-1},{\rm id}_{H_3},\dots,{\rm id}_{H_m}).$$
The following diagram is helpful in following the argument:
\[
\begin{tikzpicture}
\matrix (m) [matrix of math nodes,
             row sep=0.5pt,
             column sep=0.5pt] {
G_1&\times &G_2&\times & H_3 &\times \cdots \times& H_m & \xrightarrow{~~\quad B_2 \quad~~} & F \\
\mathllap{\scriptstyle {\rm id}_{G_1}}\bigg\downarrow & & \mathllap{\scriptstyle I_{2}^{-1}}\bigg\downarrow & & \mathllap{\scriptstyle {\rm id}_{H_3}}\bigg\downarrow & \vdots&  \mathllap{\scriptstyle {\rm id}_{H_m}}\bigg\downarrow &  & \\
G_1&\times& H_2&\times &H_3 &\times \cdots \times& H_m \\
};

\draw[->] (m-3-7) -- node[midway, below right] {$ \scriptstyle B_1$}(m-1-9);
\end{tikzpicture}\,.
\]
We have $S_2 \in \mathcal{C}(E_2,G_2)$ and $B_2 \in \mathcal{N}(G_1,G_2,H_3,\dots, H_m; G)$. From $\| B_2 \circ ({\rm id}_{G_1},h_2,T_3,\dots,T_m) \| = \| B_1\circ ({\rm id}_{G_1},h_2,T_3,\dots,T_m) \| \leq \| h_2 \|_2$ we get $\| B_2 \circ ({\rm id}_{G_1}, {\rm id}_{G_2},T_3,\dots,T_m) \| \leq 1$. And from
$$\| S_2(x_2)\|_2 = \max \{ \| S_2(x_2)\|, \| B_1 \circ({\rm id}_{G_1},T_2(x_2),T_3,\dots,T_m) \| \} \leq  \| x_2 \|$$
 for every $x_2 \in E_2$ we obtain $\| S_2 \| \leq 1$.

Repeating the procedure we construct operators $S_i \in \mathcal{C}(E_i; G_i)$ with $\| S_i \| \leq 1$, $i = 2, \dots, m$, and $B_m \in \mathcal{N}(G_1, \dots, G_m; G)$ with $\| B_m \| \leq 1$ such that $A = S_0 \circ B_m \circ (S_1,\dots,S_m)$. In particular,
$$\| A \| \leq \| S_0 \|\!\cdot\! \|B_m \| \!\cdot\!\| S_1 \| \cdots \| S_m\| \leq \| A \|.$$
Thus far we have a factorization $A=S_0\circ B_m \circ(S_1,\dots,S_m)$ so that $\|A\|=\|S_0\|\!\cdot\!\|B_m\|\!\cdot\!\|S_1\|\cdots\|S_m\|$, where $S_0 \in \mathcal{D}(G; F)$, $B_m \in \mathcal{N}(G_1, \dots, G_m; G)$, $S_i \in \mathcal{C}(E_i, G_i)$, $i = 1, \dots, m$. Putting $\alpha=\|A\|^\frac{1}{m+2}$ and taking $U_0=\frac{\alpha S_0}{\|S_0\|} \in \mathcal{D}(G; F)$, $V=\frac{\alpha B_m}{\|B_m\|} \in \mathcal{N}(G_1, \dots, G_m; G)$ and $U_i=\frac{\alpha S_i}{\|S_i\|}\in\mathcal{C}(E_i, G_i)$, $i = 1, \dots, m$, we obtain the desired factorization.

\medskip

\noindent{\bf Claim 3.} The usual norm $\|\cdot\|$ makes $\mathcal{A}$ a $\frac{1}{m+2}$-Banach space.

By Claim 1 we know that $\cal A$ is linear space. Since $\|\cdot\|$ is a norm on $\cal A$, it is a $\frac{1}{m+2}$-norm by Lemma \ref{2hjl}(a). Completeness is all that is left to be proved to establish the claim. Let $(A^n)_n$ be a sequence in $\cal A$ such that $\sum\limits_{n=1}^\infty \|A^n\|^\frac{1}{m+2} < \infty$. The norm on ${\cal L}(E_1, \ldots, E_m;F)$ is a $\frac{1}{m+2}$-norm by Lemma \ref{2hjl}(a). As ${\cal L}(E_1, \ldots, E_m;F)$ is a Banach space, it is also $\frac{1}{m+2}$-Banach space by Lemma \ref{2hjl}(b). The convergence $ \sum\limits_{n=1}^\infty \|A^n\|^\frac{1}{m+2} < \infty$ gives, together with Lemma \ref{2hjl}(c), that there exists $A\in\mathcal{L}(E_1,\dots,E_m;F)$ such that $A=\sum\limits_{n=1}^\infty \|A^n\|^\frac{1}{m+2}$, that is $\left\|\sum\limits_{j=1}^n A^n - A\right\|^{\frac{1}{m+2}} \stackrel{n \to \infty}{\longrightarrow} 0$. Of course, the convergence also occurs with respect to the usual norm. It is enough to show that $A\in\mathcal{A}$. By Claim 2, for each $n$ we can write $A^n=S^n_0\circ B^n\circ(S_1^n,\dots,S_m^n)$ with \begin{equation}\label{ykwz}\|A^n\|^\frac{1}{m+2}=\|S_0^n\|=\|B^n\|=\|S_1^n\|=\dots=\|S_m^n\|,
\end{equation} where $G_1^n, \ldots, G_m^n, G^n$ are Banach spaces, $S_0^n \in \mathcal{D}(G^n; F)$, $B^n \in \mathcal{N}(G_1^n, \dots, G_m^n; G^n)$, $S_i^n \in \mathcal{C}(E_i, G_i^n)$ for $i = 1, \dots, m$. Set $G=\left(\bigoplus\limits_{n=1}^\infty G^n\right)_1$ and $ G_i=\left(\bigoplus\limits_{n=1}^\infty G_i^n\right)_1$, $i = 1, \dots, m$. For every $n$, by $J^n\colon G^n\to G$ and $J_i^n\colon G_i^n\to G_i$ we denote the canonical inclusions, which are isometric embeddings, and by $P^n\colon G\to G^n$ and $P_i^n\colon G_i\to G_i^n$ the canonical projections, which are surjective operators. Note that, for $n \in \mathbb{N}$ and $i \in \{1, \ldots,m\}$,
\begin{equation}\label{j87b}\|J_i^n\circ S_i^n\|=\|S_i^n\|= \|J^n\circ B^n \circ (P_1^n,\dots,P_m^n)\|\leq\|B^n\| \mbox{ and } \|S_0^n\circ P^n\|\leq\|S_0^n\|.
\end{equation}
Consider the maps
$$S_i=\sum\limits_{n=1}^\infty J_i^n\circ S_i^n\colon E_1\to G_i~,~ S_i(x_i) =:(S_i^1(x_i),S_i^2(x_i),\dots), i = 1, \ldots,m,$$
\begin{align*}
B=\sum\limits_{n=1}^\infty &J^n\circ B^n\circ(P_1^n,\dots,P_m^n)\colon G_1\times\cdots\times G_m\to G \mbox{ given by}\\
&B(g_1,\dots,g_m):= (B^1(g_1^1,\dots,g_m^1),B^2(g_1^2,\dots,g_m^2),\dots), \mbox{and}
\end{align*}
\[S=\sum\limits\limits_{n=1}^\infty S_0^n\circ P^n\colon G\to F~,~S(g^1,g^2,\dots):= \sum\limits_{n=1}^\infty S_0^n(g^n).\]
Using (\ref{ykwz}), (\ref{j87b}), the convergence of the series $\sum\limits_{n=1}^\infty \|A^n\|^\frac{1}{m+2}$, that ${\cal C}(E_i)$ are injective classes of linear operators, $\cal N$ is an $m$-semi-ideal and ${\cal D}(F)$ is a surjective class of linear operators (in particular, the corresponding subspaces are closed), the maps above are well defined, each $S_i\in\mathcal{C}(E_i,G_i)$, $S\in\mathcal{D}(G,F)$ and $B\in\mathcal{N}(G_1, \dots, G_m; G)$. A routine computations shows that $A=S\circ B\circ(S_1,\dots,S_m)$, hence $A \in {\cal A}$. Claim 3 has been established.

By Claim 3, $\mathcal{A}$ is a $\frac{1}{m+2}$-Banach space. By Lemma \ref{2hjl}(b) it follows that $\cal A$ is a Banach space, hence it is a closed subspace of $\mathcal{L}(E_1,\dots,E_m;F)$.
\end{proof}

Calling ${\cal L}$ the class of all continuous multilinear operators Banach spaces, and keeping the notation of ${\cal D}(F), {\cal C}(E_1), \ldots, {\cal C}(E_m)$, the following two consequences of the theorem above are clear.

\begin{corollary} The subspace $\mathcal{D}(F)\circ\mathcal{L}\circ(\mathcal{C}(E_1),\dots,\mathcal{C}(E_m))$ is
closed in $\mathcal{L}(E_1,\dots,E_m;F)$.
\end{corollary}

\begin{corollary} Let $\mathcal{I}_0,{\cal I}_1, \dots,\mathcal{I}_m$ be closed operator ideals, and let $\mathcal{M}$ be a closed ideal of multilinear operators. Then, for all Banach spaces $E_1, \ldots, E_m,F$,  $$\mathcal{I}_0\circ\mathcal{M}\circ(\mathcal{I}_1,\dots,\mathcal{I}_m)(E_1, \ldots, E_m,F)$$ is a closed subspace of $\mathcal{L}(E_1,\dots,E_m;F)$.
\end{corollary}

Let $\mathcal{N}$ be an $m$-semi-ideal, let $\mathcal{C}_1, \dots,\mathcal{C}_m$ be injective classes of linear operators on the Banach spaces $E_1, \ldots, E_m$, respectively. For any Banach space $F$, by
$$\mathcal{N}\circ(\mathcal{C}_1, \ldots, \mathcal{C}_m)(E_1, \ldots, E_m;F)$$
we mean the set of operators $A\in\mathcal{L}(E_1,\dots,E_m;F)$ for which there are Banach spaces $G_1,\dots,G_m$, linear operators $S_i\in\mathcal{C}_i(E_i,G_i),i=1,\dots,m$, and an $m$-linear operator $B\in\mathcal{N}(G_1,\dots,G_m;F)$ such that $A=B\circ(S_1,\dots,S_m)$, that is the diagram below commutes.
\[
\begin{tikzpicture}
\matrix (m) [matrix of math nodes,
             row sep=0.5pt,
             column sep=0.5pt] {
E_1 &\times \cdots \times& E_m & \xrightarrow{\quad A \quad} & F \\
\mathllap{\scriptstyle S_{1}}\bigg\downarrow & \vdots&  \mathllap{\scriptstyle S_{m}}\bigg\downarrow &  & \\
G_1 &\times \cdots \times& G_m \\
};

\draw[->] (m-3-3) -- node[midway, below right] {$ \scriptstyle B$}(m-1-5);
\end{tikzpicture}\,.
\]

The next statement employs the symbology we have just introduced.

\begin{proposition} \label{4gim} For every Banach space $F$, $\mathcal{N}\circ(\mathcal{C}_1, \ldots, \mathcal{C}_m)(E_1, \ldots, E_m;F)$ is a closed subspace of ${\cal L}(E_1, \ldots, E_m;F)$. In particular, $\mathcal{M}\circ(\mathcal{C}_1, \ldots, \mathcal{C}_m)(E_1, \ldots, E_m;F)$ is a closed subspace of ${\cal L}(E_1, \ldots, E_m;F)$ for every closed ideal of multilinear operators $\cal M$.
\end{proposition}

\begin{proof} First let us remark that the second assertion follows immediately from Theorem \ref{Prop-8.3}. This is the case because, denoting by ${\cal L}(F)$ the class of all continuous linear operators from $F$ to $F$, we have
\begin{equation}\label{oj8r}\mathcal{M}\circ(\mathcal{C}_1, \ldots, \mathcal{C}_m)(E_1, \ldots, E_m;F) = \mathcal{L}(F)\circ\mathcal{M}\circ(\mathcal{C}_1(E_1),\dots,\mathcal{C}_m(E_m)),
\end{equation}
where the latter space is the one from Theorem \ref{Prop-8.3}. The point is that (\ref{oj8r}) is no longer true if the closed ideal of multilinear operators $\cal M$ is replaced with an $m$-semi-ideal $\cal N$, so the case where $\cal N$ is an $m$-semi-ideal demands its own proof.

Now let us see that it takes only a few adjustments in the proof of Theorem \ref{Prop-8.3} to get the first assertion. Instead of reworking the whole proof, we just point out the necessary adjustments. Call $\mathcal{B}= \mathcal{N}\circ(\mathcal{C}_1, \ldots, \mathcal{C}_m)(E_1, \ldots, E_m;F)$.

In the proof of Claim 1 it is enough to disregard the spaces and operators with subscripts $0$ to get that $\mathcal{B}$ is a linear subspace.

The statement of Claim 2 is the following: Every $A\in\mathcal{B}$ admits a factorization $A=V\circ(U_1,\dots,U_m)$ such that $\|A\|^\frac{1}{m+1}=\|V\|=\|U_1\|=\dots=\|U_m\|$, where $V \in \mathcal{N}(G_1, \dots, G_m; G)$ and $U_i \in \mathcal{C}_i(E_i, G_i)$ for $i = 1, \dots, m$. To prove the claim, start with a factorization $A = B \circ (T_1, \ldots, T_m)$ and note that $A={\rm id}_F\circ B\circ(T_1,\dots,T_m)$. Use that, for every Banach space $G$ and any bijective operator (isomorphism) $I\in\mathcal{L}(G,F)$, it holds $I\circ B\in\mathcal{N}$. In the original proof of Claim 2, replace the operator $T_0$ with the operator ${\rm id}_F$ and proceed as in there. The procedure will result in a factorization $A = B_m \circ (S_1, \ldots, S_m)$ so that $\|A\| = \|B_m\|\!\cdot\!\|S_1\|\cdots\|S_m\|$. To complete the proof of the claim, put $\alpha=\|A\|^\frac{1}{m+1}$ and take $V=\frac{\alpha B_m}{\|B_m\|}$,  $U_i=\frac{\alpha S_i}{\|S_i\|}, i = 1, \dots, m$.

The statement of Claim 3 is the following: $\mathcal{B}$ is a  $\frac{1}{m+1}$-Banach space. Disregard the spaces $G^n$'s and the operators with subscript 0 in the original proof of Claim 3. Use the same argument as in there to conclude that it is enough to check that $\cal B$ is complete regarding the usual norm as a $\frac{1}{m+1}$-norm. Given a sequence $(A^n)_n$ in $ \mathcal{B}$ with $\sum\limits_{n=1}^\infty \|A^n\|^\frac{1}{m+1} < \infty$, from Claim 2, for each $n$ writhe $A^n=B^n\circ(S_1^n,\dots,S_m^n)$ so that $$\|A^n\|^\frac{1}{m+1}=\|B^n\|=\|S_1^n\|=\dots=\|S_m^n\|,$$ where $B^n \in \mathcal{N}(G_1^n, \dots, G_m^n; F)$ and $S_i^n \in \mathcal{C}_i(E_i, G_i^n)$ for $i = 1, \dots, m$. Now repeat the steps of the original proof to get the claim and to complete the proof of the result. 
\end{proof}

\begin{example}\rm Recall that a continuous linear operator $T$ from a Banach lattice $E$ to a Banach space $X$ is {\it almost Dunford-Pettis} if $T$ sends disjoint weakly null sequences in $E$ to norm null sequences in $X$. This class was introduced by S\'anchez \cite{sanchez}, significant developments
appeared in \cite{aqzzouz, wnuk1994}, for recent contributions, see \cite{mik, ardvin, ardvali, khabaoui}.

By $M$-$W(E)$ and $ADP(E)$ we denote the classes of $M$-weakly compact and almost Dunford-Pettis linear operators, respectively, from a Banach lattice $E$ to a Banach space. For every Banach space $F$, $M$-$W(E,F)$ and $ADP(E,F)$ are closed subspaces of ${\cal L}(E,F)$; indeed, this follows from the linear case $m = 1$ of Proposition \ref{fechados-seq-prop}. Furthermore, it is clear that  $M$-$W(E)$ and $ADP(E)$ are injective classes of linear operators on $E$.

For Banach lattices $E_1,\ldots, E_m$ and a Banach space $F$, for every choice of ${\cal C}_j\in \{M\mbox{-}W, ADP\}$, $j = 1, \ldots, m$, ${\cal L}\circ ({\cal C}_1, \ldots, {\cal C}_m)(E_1, \ldots, E_m;F) $
is a closed subspace of ${\cal L}(E_1, \ldots, E_m;F)$ by Proposition \ref{4gim}. In particular, ${\cal L}\circ (M\mbox{-}W, \ldots, M\mbox{-}W)(E_1, \ldots, E_m;F) $
is a closed subspace. For an application of the class of multilinear operators ${\cal L}\circ (M\mbox{-}W, \ldots, M\mbox{-}W)$, see \cite[Proposition]{Ariel}.
\end{example}

\begin{corollary} Let $\mathcal{I},\mathcal{I}_1,\dots,\mathcal{I}_m$ be closed operator ideals and let $\mathcal{M}$ be a closed ideal of multilinear operators. For all Banach spaces $E_1, \ldots, E_m$ and $F$,  the sets $\mathcal{M}\circ(\mathcal{I}_1,\dots,\mathcal{I}_m)(E_1, \ldots, E_m;F)$, $\mathcal{L}\circ(\mathcal{I}_1,\dots,\mathcal{I}_m)(E_1, \ldots, E_m;F)$ and $\mathcal{L}\circ(\mathcal{I},\dots,\mathcal{I})(E_1, \ldots, E_m;F)$ are all closed subspaces of $\mathcal{L}(E_1,\dots,E_m;F)$.
\end{corollary}

\begin{examples}\rm By $\cal K$ and ${\cal W}$ we denote the closed ideals of compact and weakly compact linear operators, respectively. By the corollary above we know that, for all Banach spaces $E_1, \ldots, E_m$ and $F$, ${\cal L}\circ ({\cal K},\stackrel{(m)}{\ldots} ,{\cal K})(E_1, \ldots, E_m;F)$ and  ${\cal L}\circ ({\cal W},\stackrel{(m)}{\ldots} ,{\cal W})(E_1, \ldots, E_m;F)$ are closed subspaces of $\mathcal{L}(E_1,\dots,E_m;F)$. The classes of multilinear operators ${\cal L}\circ ({\cal K},\stackrel{(m)}{\ldots} ,{\cal K})$ and ${\cal L}\circ ({\cal W},\stackrel{(m)}{\ldots} ,{\cal W})$ were investigated extensively, see, e.g., \cite{AG, AHV, Note, BJ, GG}. For instance: \\
$\bullet$ Several useful characterizations of the multilinear operators belonging to ${\cal L}\circ({\cal W},\ldots ,{\cal W})$ were proved in  \cite{AG}.\\
$\bullet$ In \cite{AHV} it was proved that ${\cal L}\circ ({\cal K},\ldots ,{\cal K})$ coincides with the class of multilinear operators that are weakly continuous on bounded sets, a class of multilinear operators that has been studied for a long time, for recent developments see \cite{Sergio2022, Ewerton, BJW, Sergio2017, Sergio2018, Sergio2025}.
\end{examples}

\section{Final remarks}

\indent\indent Just for record, we state in this section two further results in the line of using closed subspaces of linear operators to produce closed subspaces of multilinear operators that might be of interest. Let the Banach spaces $E_1, \ldots, E_m$ and $F$ be given.

By $E_1 \widehat{\otimes}_\pi \cdots \widehat{\otimes}_\pi E_m$ we denote the completed projective tensor product of $E_1, \ldots, E_m$. For every $A \in {\cal L}(E_1, \ldots, E_m;F)$ there is a unique linear operator $A_L \colon E_1 \widehat{\otimes}_\pi \cdots \widehat{\otimes}_\pi E_m \to F$ such that $A_L(x_1 \otimes \cdots \otimes x_n) = A(x_1, \ldots, x_m)$ for every $(x_1, \ldots, x_m) \in E_1 \times \cdots \times E_m$. Moreover, the correspondence
$$A \in  {\cal L}(E_1, \ldots, E_m;F) \mapsto A_L \in {\cal L}(E_1 \widehat{\otimes}_\pi \cdots \widehat{\otimes}_\pi E_m,F) $$
is an isometric isomorphism \cite{df, ryan}.

\begin{proposition} Let $M$ be a closed subspace of ${\cal L}(E_1 \widehat{\otimes}_\pi \cdots \widehat{\otimes}_\pi E_m,F)$. Then the set of multilinear operators $A \in  {\cal L}(E_1, \ldots, E_m;F)$ so that $A_L \in M$ is a closed subspace of ${\cal L}(E_1, \ldots, E_m;F)$.
\end{proposition}

For $j = 1, \ldots, m$, the symbol $\stackrel{[j]}{\ldots}$ means that the $j$-th coordinate has been omitted. The  canonical operator $I_j \colon {\cal L}(E_1, \ldots, E_m;F) \to {\cal L}(E_j, {\cal L}(E_1, \stackrel{[j]}{\ldots}, E_m;F ))$ given by
$$I_j(A)(x_j)(x_1, \stackrel{[j]}{\ldots}, x_m) = A(x_1, \ldots, x_j, \ldots, x_m), $$
is an isometric isomorphism for every $j \in \{1, \ldots, m\}$ \cite{Din, mujica}.

\begin{proposition} Choose $j_1, \ldots, j_k \in \{1, \ldots, m\}$ and let $M_{j_i}$ be a closed linear subspace of ${\cal L}(E_{j_i}, {\cal L}(E_1, \stackrel{[j_i]}{\ldots}, E_m;F ))$, $i = 1, \ldots, k$.  Then the set of multilinear operators $A \in  {\cal L}(E_1, \ldots, E_m;F)$ so that $I_{j_i}(A) \in M_{j_i}$ for every $i = 1,\ldots, k$, is a closed subspace of ${\cal L}(E_1, \ldots, E_m;F)$.
\end{proposition}

\bigskip
\noindent Geraldo Botelho~~~~~~~~~~~~~~~~~~~~~~~~~~~~~~~~~~~~~~Ariel Mon\c c\~ao\\
Instituto de Matem\'atica e Estat\'istica~~~~~~~~~\,\,\,Departamento de Matemática\\
Universidade Federal de Uberl\^andia~~~~~~~~~~~~~Universidade Federal de Minas Gerais\\
38.400-902 -- Uberl\^andia -- Brazil~~~~~~~~~~~~~~~~~31.270-901 -- Belo Horizonte -- Brazil\\
e-mail: botelho@ufu.br~~~~~~~~~~~~~~~~~~~~~~~~~\,~~~~~e-mail: arieldeom@hotmail.com


\begin{thebibliography}{99}\small

\vspace{-0.5em}

\bibitem{mik} Aires, M., Botelho, G., {\it Spaces of sequences not converging to zero},  Ann. Fenn. Math. {\bf 51} (2026), no. 1, 41-58.

\vspace{-0.5em}

\bibitem{Alikhani} Alikhani, M., {\it On pseudo weakly compact operators of order $p$}, https://doi.org/10.48550/ arXiv.1810.05638, 2018.

\vspace{-0.5em}

\bibitem{Alip} Aliprantis, C. D.,  Burkinshaw, O., {\it Positive Operators}, Springer, Dordrecht, 2006.

\vspace{-0.5em}

\bibitem{Alpay} Alpay, Ş., Ercan, Z., {\it Characterizations of Riesz spaces with b-property}, Positivity {\bf 13}, 21–30 (2009). 

\vspace*{-0.5em}
\bibitem{aqzzouz} B. Aqzzouz, A. Elbour, {\it Some characterizations of almost Dunford–Pettis operators and applications}, Positivity {\bf 15} (2011), 369-380.


\vspace{-0.5em}

\bibitem{Aqzzouz} Aqzzouz, B., H'Michane, J. {\it The class of b-AM-compact operators}, Quaest. Math. {\bf 36} (2013), no. 3, 309-319.




\vspace{-0.5em}



\bibitem{ardvin} Ardakani, H., Miranda, V. C. C., {\it Dunford–Pettis like sets with applications to spaces of
operators}, Bull. Braz. Math. Soc. (N.S.) {\bf 56} (2025), Paper No. 18, 18 pp.

\vspace{-0.5em}

\bibitem{ardvali} Ardakani, H., Vali, F., {\it On almost limited $p$-convergent operators on Banach lattices}, Positivity {\bf 28} (2024), Paper No. 20, 17 pp.

\vspace{-0.5em}

\bibitem{livro} Aron, R. M., Bernal González, L., Pellegrino, D. M., Seoane Sepúlveda, J. B., {\it  Lineability: the search for linearity in mathematics},  CRC Press, Boca Raton, 2016.

\vspace{-0.5em}

\bibitem{Aron2} Aron, R., Dimant, V. {\it Sets of weak sequential continuity for polynomials}, Indag. Math. (N.S.) {\bf 13} (2002), 287-299.

\vspace{-0.5em}

\bibitem{AG} Aron, R. M., Galindo, P., {\it Weakly compact multilinear mappings}, Proc. Edinburgh Math. Soc. (2) {\bf 40} (1997), no. 1, 181-192.

\vspace{-0.5em}

\bibitem{AHV} Aron, R. M., Herv\'es, C., Valdivia, M., {\it Weakly continuous mappings on Banach spaces}, J. Functional Analysis {\bf 52} (1983), 189-204.

\vspace{-0.5em}

\bibitem{Aron} Aron, R. M., Schottenloher, M., {\it Compact holomorphic mappings on Banach spaces and the approximation property}. J. Functional Analysis {\bf 21} (1976), no. 1, 7-30.

\vspace{-0.5em}

\bibitem{Baklouti} Baklouti, H., Hajji, M., Moulahi, R., {\it $b$-AM-Dunford–Pettis Operators on Banach lattices}, Complex Anal. Oper. Theory {\bf 18}, 77 (2024).

\vspace{-0.5em}

\bibitem{Note} Botelho, G., {\it Ideals of polynomials generated by weakly compact operators}, Note Mat. {\bf 25} (2005/06), no. 1, 69-102.
\vspace{-0.5em}

\bibitem{Sergio2022} Botelho, G., F\'avaro, V. V., P\'erez, S. A., {\it Uncomplemented subspaces in operator and polynomial ideals}, Rev. Mat. Complut. {\bf 35} (2022), no. 3, 851-869.

\vspace{-0.5em}

\bibitem{JLPAMS} Botelho, G., Luiz, J. L. P., {\it The positive polynomial Schur property in Banach lattices}, Proc. Amer. Math. Soc. {\bf 149} (2021), 2147-2160.

\vspace{-0.5em}

\bibitem{Vinger} Botelho, G., Miranda, V. C. C., {\it Compact positive multilinear operators on Banach lattices}, Positivity {\bf 30}, 20 (2026).

\vspace{-0.5em}

\bibitem{Ariel}  Botelho, G., Mon\c c\~ao, A., {\it $L$- and $M$-weakly compact multilinear operators and their linear adjoints}, Indag. Math. (N.S.), to appear, https://doi.org/10.1016/j.indag.2026.02.004, 2026.

\vspace{-0.5em}

\bibitem{Botelho} Botelho, G., Pellegrino, D., Teixeira, E., {\it Introduction to Functional Analysis}, Springer, Cham, 2025.

\vspace{-0.5em}

\bibitem{Ewerton2} Botelho, G., Torres, E. R., {\it Hyper-ideals of multilinear operators}, Linear Algebra Appl. {\bf 482} (2015), 1-30.

\vspace{-0.5em}

\bibitem{Ewerton} Botelho, G., Torres, E. R., {\it Strongly factorable multilinear operators on Banach spaces}, Colloq. Math. {\bf 154} (2018), no. 1, 15-30.

\vspace{-0.5em}

\bibitem{Bourbaki} Bourbaki, N., {\it Topological Vector Spaces}, Elements of Mathematics, Springer, 2003.

\vspace{-0.5em}

\bibitem{Braunss} Braunss, H. A., {\it Multi-ideals with special properties}, Blatter Potsdamer Forschungen 1/87, Potsdam, 1987.


\vspace{-0.5em}

\bibitem{BJ} Braunss, H. A., Junek, H., {\it Factorization of injective ideals by interpolation}, Special issue dedicated to John Horv\'ath, J. Math. Anal. Appl. {\bf 297} (2004), no. 2, 740-750.


\vspace{-0.5em}

\bibitem{Buskes3} Bu, Q., Buskes, G., {\it Polynomials on Banach lattices and positive tensor products}, J. Math. Anal. Appl. {\bf 388} {\bf 2}, (2012) 845-862.

\vspace{-0.5em}

\bibitem{BJW} Bu, Q., Ji, D., Wong, N-C., {\it Weak sequential completeness of spaces of homogeneous polynomials}, J. Math. Anal. Appl. {\bf 427} (2015), no. 2, 1119-1130.

\vspace{-0.5em}

\bibitem{Castillo1} Castillo, J. M. F., Garc\'ia, R., Gonzalo, R., {\it Banach spaces in which all multilinear forms are weakly sequentially continuous}, Studia Math. {\bf 136} (1999), 121-145.

\vspace{-0.5em}

\bibitem{Castillo2} Castillo, J. M. F., Garc\'ia, R.,  Gonzalo, R., {\it Stability properties of the class of Banach spaces in which all multilinear forms are weakly sequentially continuous}, Glasgow Math. J. {\bf 44} (2002), 81-92.

\vspace{-0.5em}

\bibitem{sheldon} Dantas, S., Medina, R. {\it On holomorphic functions attaining their weighted norms}, Rev. R. Acad. Cienc. Exactas Fís. Nat., Ser. A Mat., RACSAM {\bf 119} (2025), no. 1, Paper no. 14, 22 p.

\vspace{-0.5em}

\bibitem{df} Defant, A., Floret, K., {\it Tensor Norms and Operator Ideals}, North-Holland, 1993.

\vspace{-0.5em}

\bibitem{Din} Dineen, S., {\it Complex Analysis on Infinite Dimensional Spaces}, Springer, London, 1999.

\vspace{-0.5em}

\bibitem{fg} Floret, K., García, D., {\it On ideals of polynomials and multilinear mappings between Banach spaces}, Arch. Math. (Basel) {\bf 81} (2003), no. 3, 300-308.

\vspace{-0.5em}

\bibitem{Fremlin2} Fremlin, D. H., {\it Tensor products of Archimedean vector lattices}, Amer. J. Math. {\bf 94} (3) (1972) 777-798.

\vspace{-0.5em}

\bibitem{Fremlin1} Fremlin, D. H., {\it Tensor products of Banach lattices}, Math. Ann. {\bf 211} (1974) 87-106.

\vspace{-0.5em}

\bibitem{VP} Galindo, P., Miranda, V., {\it A class of sets in a Banach space coarser than limited sets}, Bull. Braz. Math. Soc. {\bf 53} (2022), 941-955.

\vspace{-0.5em}

\bibitem{Garling} Garling, D. J.H., {\it Inequalities: A Journey into Linear Analysis}, Cambridge University Press, 2007.



\vspace{-0.5em}

\bibitem{GG} Gonz\'alez, M., Guti\'errez, J. M., {\it Factorization of weakly continuous holomorphic mappings}, Studia Math. {\bf 118} (1996), no. 2, 117-133.

\vspace{-0.5em}

\bibitem{Hajek} H\'ajek, P., Johanis, M. {\it Smooth Analysis in Banach Spaces}, Series in Nonlinear Analysis and Applications 19, De Gruyter, Berlin, 2014.

\vspace{-0.5em}

\bibitem{Hajji} Hajji, M., Mahfoudhi, M., {\it LW-compact operators and domination problem}, Positivity {\bf 25}, 1959–1972 (2021).

\vspace{-0.5em}

\bibitem{Jin} Jin Xi, C., Jingge F., {\it $DW$-compact operators on Banach lattices}, Positivity {\bf 29} (2025), no.1, Paper No. 14, 15 pp.


\vspace{-0.5em}

\bibitem{kalton} Kalton, N., {\it Quasi-Banach spaces}, Handbook of the geometry of Banach spaces, Vol. 2, 1099-1130, North-Holland, Amsterdam, 2003.

\vspace*{-0.5em}
\bibitem{khabaoui} Khabaoui, H., H'michane, J., El Fahri, K., {\it A contribution to operators between Banach lattices}, Positivity {\bf 28} (2024), no. 5, Paper No. 65, 11 pp.

\vspace{-0.5em}

\bibitem{kusraeva} Kusraeva, Z. A., {\it Sums of disjointness preserving multilinear operators}, Positivity {\bf 25} (2021) 669-678.


\vspace{-0.5em}

\bibitem{lewis} Lewis, P., {\it Dunford-Pettis sets}, Proc. Amer. Math. Soc. {\bf 129} (2001), 3297-3302.

\vspace{-0.5em}

\bibitem{loane} Loane, J., {\it Polynomials on vector lattices}, Ph.D. Thesis, National University of Ireland, Galway, 2007.

\vspace{-0.5em}

\bibitem{megginson} Megginson, R. E., {\it An Introduction to Banach Space Theory}, Springer, 1998.

\vspace{-0.5em}

\bibitem{Meyer-Peter} Meyer-Nieberg, P., {\it Banach Lattices}, Springer-Verlag, 1991. 

\vspace{-0.5em}

\bibitem{mu} Mujica, J., {\it Linearization of bounded holomorphic mappings on Banach spaces}, Trans. Amer. Math. Soc. {\bf 324} (1991) 867-887.

\vspace*{-0.5em}

\bibitem{mujica} J. Mujica, {\it Complex Analysis in Banach Spaces}, Dover Publications, 2010.


\vspace{-0.5em}

\bibitem{arXiv2026} Ondrej, S., Spurn\'y, J., {\it Operators on injective tensor products of $L_1$-preduals}, arXiv:2604.18157v1, 2026.

\vspace*{-0.5em}
\bibitem{Peralta} Peralta, A. M., Villanueva, I., Wright, J. D. M., Ylinen, K., {\it Topological characterisation of weakly compact operators}, J. Math. Anal. Appl. {\bf 325} no. 2 (2007), 968-974.


\vspace{-0.5em}





\bibitem{Sergio2017} P\'erez, S. A., {\it On the reflexivity of ${\cal P}_w(^nE;F)$}, Arch. Math. (Basel) {\bf 109} (2017), no. 5, 471-475.

\vspace{-0.5em}

\bibitem{Sergio2018} P\'erez, S. A., {\it Complemented subspaces of homogeneous polynomials}, Rev. Mat. Complut. {\bf 31} (2018), no. 1, 153-161.



\vspace{-0.5em}

\bibitem{Sergio2025} P\'erez, S. A., Rinc\'on-Villamizar, M. A., {\it Reflexivity and weak sequential completeness in operator ideals and polynomial ideals}, Bull. Braz. Math. Soc. (N.S.) {\bf 56} (2025), no. 4, Paper No. 58, 15 pp.

\vspace{-0.5em}

\bibitem{Pietsch} Pietsch, A., {\it Operators Ideals}, North-Holland, 1980.


\vspace{-0.5em}

\bibitem{pietsch83} Pietsch, A., {\it Ideals of multilinear functionals (designs of a theory)}, Proceedings of the second international conference on operator algebras, ideals, and their applications in theoretical physics (Leipzig, 1983), 185-199, Teubner-Texte Math., 67, Teubner, Leipzig, 1984.

\vspace{-0.5em}

\bibitem{Ribeiro} Ribeiro, J., Santos, F., Torres, E. R., {\it Coherence and compatibility: a stronger approach}, Linear Multilinear Algebra, {\bf 70} (2022), no. 1, 66–80. 

\vspace{-0.5em}

\bibitem{ry} Ryan, R., {\it Weakly compact holomorphic mappings on Banach spaces}, Pacific J. Math. {\bf 131} (1988) 179-190.

\vspace*{-0.5em}

\bibitem{ryan} Ryan, R. A., {\it Introduction to Tensor Products of Banach Spaces}, Springer, 2002.

\vspace*{-0.5em}
\bibitem{sanchez} S\'anchez, J. A., {\it Operators on Banach lattices} (Spanish), Ph. D. Thesis, Universidad Complutense de Madrid, 1985.

\vspace{-0.5em}

\bibitem{Schaefer} Schaefer, H. H., {\it Banach Lattices and Positive Operators}, Springer Berlin, Heidelberg, 1974.

\vspace{-0.5em}

\bibitem{Schaefer2} Schaefer, H. H., Wolff, M.P., {\it Topological Vector Spaces}, Second Edition, Springer, 1999.

\vspace{-0.5em}

\bibitem{wnuk1994} Wnuk, W., {\it Banach lattices with the weak Dunford-Pettis property}, Atti Sem. Mat. Fis. Univ.
Modena {\bf 42} (1994), no. 1, 227-236.


\end{thebibliography}
\end{document}